\documentclass[11pt, a4paper]{amsart}
\usepackage{amssymb, array, amsmath, amscd, pdfpages, enumerate, amsthm, fixltx2e, setspace, mathtools}

\usepackage[utf8]{inputenc}
\usepackage{amssymb}
\usepackage{bm}
\usepackage{graphicx}

\newcommand{\gtrgtr}{\mathbin{\rotatebox[origin=c]{180}{$\ll$}}}

\newcommand{\yref}{y_{\mbox{\scriptsize\textit{ref}}}}
\newcommand{\wdist}{w_{\mbox{\scriptsize\textit{dist}}}}
\newcommand{\wdistdot}{\dot{w}_{\mbox{\scriptsize\textit{dist}}}}
\newcommand{\wext}{w_{\mbox{\scriptsize\textit{ext}}}}

\newcommand{\mc}[1]{\mathcal{#1}}
\newcommand{\GG}{{\mathcal G}}
\newcommand{\Hinf}{\mc{H}_\infty}
\newcommand{\glacc}{\gl_{\mbox{\footnotesize acc}}}
\newcommand{\citel}[2]{\cite[#2]{#1}}

\newcommand{\Acomp}{A_s}
\newcommand{\Bcomp}{B_s}
\newcommand{\Ccomp}{C_s}
\newcommand{\Qcomp}{Q_s}

\newtheorem{thm}{Theorem}[section]

\newtheorem{lem}[thm]{Lemma}

\newtheorem{ass}[thm]{Assumption}

\theoremstyle{definition}

\newtheorem{rem}[thm]{Remark}

\newcommand{\pmat}[1]{\begin{bmatrix}#1 \end{bmatrix}}
\newcommand{\pmatsmall}[1]{\begin{bsmallmatrix}#1\end{bsmallmatrix}}

\newcommand{\bmatsmall}[1]{\begin{bsmallmatrix}#1\end{bsmallmatrix}}

\newcommand{\eps}{\varepsilon}

\DeclareMathOperator{\re}{Re}

\DeclareMathOperator{\diag}{diag}

\newcommand*{\C}{{\mathbb{C}}}     
\newcommand*{\R}{{\mathbb{R}}}

\newcommand*{\N}{{\mathbb{N}}}

\newcommand*{\Lin}{{\mathcal{L}}}   
\newcommand*{\Dom}{D}   
\newcommand{\ran}{{\mathcal{R}}}   
\renewcommand{\ker}{{\mathcal{N}}}

\newcommand*{\abs} [1]{\lvert#1\rvert}
\newcommand*{\norm}[1]{\lVert#1\rVert}
\newcommand*{\set} [1]{\{#1\}}
\newcommand*{\setm}[2]{\{\,#1\mid#2\,\}}   
\newcommand*{\iprod}[2]{\langle#1,#2\rangle}

\newcommand*{\Lp}[1][p]{L^{#1}}

\newcommand{\Abs}[2][default]{\ifthenelse{\equal{#1}{default}}{\left\lvert#2\right\rvert}{\ldelim{#1}{\lvert}#2\rdelim{#1}{\rvert}}}
\newcommand{\Norm}[2][default]{\ifthenelse{\equal{#1}{default}}{\left\lVert#2\right\rVert}{\ldelim{#1}{\lVert}#2\rdelim{#1}{\rVert}}}

\newcommand*{\Iprod}[3][default]{\ifthenelse{\equal{#1}{default}}{\left\langle#2,#3\right\rangle}{\ldelim{#1}{\langle}#2,#3\rdelim{#1}{\rangle}}}
\newcommand*{\Dualpair}[3][default]{\ifthenelse{\equal{#1}{default}}{\left\langle#2,#3\right\rangle}{\ldelim{#1}{\langle}#2,#3\rdelim{#1}{\rangle}}}

\newcommand*{\List}[2][1]{\set{#1,\ldots,#2}}

\newcommand{\eq}[1]{\begin{align*}#1\end{align*}}
\newcommand{\eqn}[1]{\begin{align}#1\end{align}}

\newcommand{\gs}{\sigma}
\newcommand{\ga}{\alpha}
\newcommand{\gb}{\beta}
\renewcommand{\gg}{\gamma}

\newcommand{\gl}{\lambda}
\newcommand{\gw}{\omega}

\newcommand{\inv}{^{-1}}

\newcommand*{\ddb}[2][1]{\ifthenelse{\equal{#1}{1}}{\frac{d}{d#2}}{\frac{d^{#1}}{d#2^{#1}}}}
\newcommand*{\pd}[3][1]{\ifthenelse{\equal{#1}{1}}{\frac{\partial{#2}}{\partial{#3}}}{\frac{\partial^{#1}{#2}}{\partial#3^{#1}}}}

\newcommand*{\keyterm}[1]{\emph{#1}}
\newtheorem*{RORP}{The Robust Output Regulation Problem}

\begin{document}

  \title[Reduced Order Controller Design for Robust Regulation]{Reduced Order Controller Design for Robust Output Regulation}

\thispagestyle{plain}

\author[L. Paunonen]{Lassi Paunonen}
\address[L. Paunonen]{Mathematics, Faculty of Information Technology and Communication Sciences, Tampere University, PO.\ Box 692, 33101 Tampere, Finland}
\email{lassi.paunonen@tuni.fi}

\author[D. Phan]{Duy Phan}
\address[D. Phan]{Mathematics, Faculty of Information Technology and Communication Sciences, Tampere University, PO.\ Box 692, 33101 Tampere, Finland}
\email{pducduy0811022@gmail.com}

\thanks{The research is supported by the Academy of Finland Grant number 310489 held by L. Paunonen. L. Paunonen is funded by the Academy of Finland Grant number 298182.}

\begin{abstract}
  We study robust output regulation for parabolic partial differential equations and other infinite-dimensional linear systems with analytic semigroups. As our main results we show that robust output tracking and disturbance rejection for our class of systems can be achieved using a finite-dimensional controller and present algorithms for construction of two different internal model based robust controllers. The controller parameters are chosen based on a Galerkin approximation of the original PDE system and employ balanced truncation to reduce the orders of the controllers. In the second part of the paper we design controllers for robust output tracking and disturbance rejection for a 1D reaction--diffusion equation with boundary disturbances, a 2D diffusion--convection equation, and a 1D beam equation with Kelvin--Voigt damping.  
\end{abstract}

\subjclass[2010]{%
93C05, 
93B52 
35K90, 
(93B28) 
}
\keywords{Robust output regulation, partial differential equation, controller design, Galerkin approximation, model reduction.} 

\maketitle

\section{Introduction}
\label{sec:intro}

In the \keyterm{robust output regulation problem} the main objective is to design a dynamic error feedback controller 
so that the output $y(t)$ of the linear infinite-dimensional system
\begin{subequations}
  \label{eq:plantintro}
  \eqn{
    \dot{x}(t)&= Ax(t) + Bu(t) + B_d\wdist(t), ~\, x(0)=x_0\in X\\
    y(t)&= Cx(t) + Du(t) + D_d\wdist(t)
  }
\end{subequations}
on a Hilbert space $X$ converges to a given reference signal $\yref(t)$  despite the external disturbance signal $\wdist(t)$, i.e.,
  \eq{
    \norm{y(t)-\yref(t)}\to 0, \qquad \mbox{as} \quad t\to\infty .
  } 
  In addition, the control is required to be robust in the sense that the designed controller achieves the output tracking and disturbance rejection even under uncertainties and perturbations in the parameters $(A,B,B_d,C,D,D_d)$ of the system (see Section~\ref{sec:RORP} for the detailed assumptions on~\eqref{eq:plantintro}).
The closed-loop system consisting of~\eqref{eq:plantintro} and a dynamic error feedback controller is depicted in Figure~\ref{fig:FBcontrol}.
  In particular the controller only uses the knowledge of the regulation error $e(t)=y(t)-\yref(t)$.

\begin{figure}[h!]
  \begin{center}
  \includegraphics[width=0.7\linewidth]{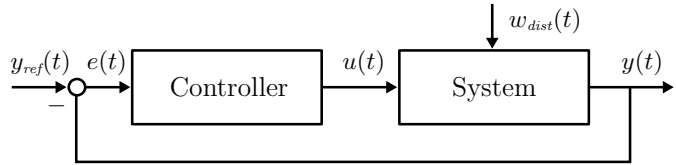}
    \caption{Dynamic error feedback control scheme.}
    \label{fig:FBcontrol}
  \end{center}
\end{figure}

The design of controllers for robust output regulation of infinite-dimen\-sional linear systems has been studied in several references~\cite{Poh82,LogTow97,HamPoh00,RebWei03,HamPoh10,Pau16,Pau17b}, and many articles also study the controller design for output tracking and disturbance rejection without the robustness requirement~\cite{Sch83a,ByrLau00,Deu11,Deu15,XuDub17a}.
In this paper we concentrate on construction of \keyterm{finite-dimensional} low-order robust controllers for control systems~\eqref{eq:plantintro} with distributed inputs and outputs. The motivation for this research arises from the fact that the robust controllers introduced in earlier references~\cite{HamPoh10,Pau16} are necessarily inifinite-dimensional unless the system~\eqref{eq:plantintro} is either exponentially stable or stabilizable by static output feedback.

As the main results of this paper we introduce two finite-dimensional controllers that solve the robust output regulation problem for possibly unstable parabolic PDE systems. The controller design is based on the \keyterm{internal model principle}~\cite{FraWon75a,Dav76,PauPoh10} which characterizes the solvability of the control problem. 
The general structures of the controllers are based on two infinite-dimensional controllers presented in~\cite{HamPoh10} and~\cite{Pau16}, respectively. Both of the infinite-dimensional controllers from~\cite{HamPoh10,Pau16} incorporate 
an observer-type copy of the original system that is used in stabilizing the closed-loop system.
In this paper 
these observer-type parts are replaced with finite-dimensional low-order systems that are constructed based on a Galerkin approximation of the system $(A,B,C,D)$ and subsequent model reduction using balanced truncation. All controller parameters are computed based on a finite-dimensional approximation of $(A,B,C,D)$ and
only involve matrix computations.
In particular, when using the Finite Element Method, both the approximation of the system~\eqref{eq:plantintro} and the model reduction step in the controller construction can be completed efficiently using existing software implementations, and this facilitates straightforward construction of our robust controllers even for complicated PDE systems. The finite-dimensional controllers introduced in this paper
can also be preferable to
the low-gain robust controllers~\cite{HamPoh00,LogTow97,RebWei03} for exponentially stable systems,
since they can typically achieve larger closed-loop stability margins and faster convergence rates of the output.

In the second part of the paper we employ the construction algorithms to design controllers for robust output regulation of selected classes of PDE models --- a 1D reaction--diffusion equation, a 2D reaction--diffusion--con\-vection equation, and a 1D beam equation with Kelvin--Voigt damping.
The general assumptions on the Galerkin approximation scheme used in the controller design have been verified in the literature for several classes of PDE models and the Finite Element approximation schemes used in this paper.

The possibility of using Galerkin approximations in the controller design is based on the theory developed in~\cite{BanKun84,Ito90,Mor94,BanIto97,ItoMor98,XiaBas99,Mor01}. 
Using Galerkin approximations in dynamic stabilization is a well-known and frequently used technique~\cite{Ito90,Mor94b,Cur03}, and in this paper we employ the same methodology in constructing finite-dimensional low-order controllers for robust output regulation. 
In the proofs of our main results we show that the closed-loop systems with our reduced order controllers approximate --- in the sense of graph topology --- closed-loop systems with infinite-dimensional controllers which can be shown to achieve closed-loop stability,
and therefore the controllers achieve
robust output regulation provided that the orders of the approximations are sufficiently high.
The graph topology was first used for the dynamic stabilization problem with Galerkin approximations in~\cite{Mor94b},
  and a detailed theoretic framework for constructing controllers based on balanced truncations was presented in~\cite{Cur03}. 
  Our proofs are especially based on the techniques in~\cite{Mor94,Mor94b}.
Controller construction for robust output regulation using Galerkin approximations was first studied in~\cite{PhaPauMTNS18} for a 1D heat equation with constant coefficients. In this paper we improve and extend the controller design method to be applicable for a larger class of control systems, include model reduction as a part of the design procedure, and consider two different controller structures.

The reference signals $\yref:\R\to \C^p$ and the disturbance signals $\wdist:\R\to \C^{2q+1}$ we consider are of the form 
\begin{subequations}
  \label{eq:yrefwdist}
  \eqn{
    \hspace{-1ex}  \yref(t) &= a_0^1(t) + \sum_{k=1}^q (a_k^1(t) \cos(\gw_k t) + b_k^1(t) \sin(\gw_k t))\\
  \label{eq:yrefwdistwd}
    \hspace{-1ex} \wdist(t) &= a_0^2(t) + \sum_{k=1}^q (a_k^2(t) \cos(\gw_k t) + b_k^2(t) \sin(\gw_k t))
  }
\end{subequations}
for some known frequencies $\set{\gw_k}_{k=0}^q\subset \R$ with $0=\gw_0<\gw_1<\ldots<\gw_q$ and \keyterm{unknown} coefficient polynomial vectors $\set{a_k^j(t)}_{k,j}$ and $\set{b_k^j(t)}_{k,j}$ with real or complex coefficients (any of the polynomials are allowed to be zero). 
We assume the maximum orders of the coefficient polynomial vectors are known, so that
  $a_k^1(t), b_k^1(t)\in \C^p$ and $a_k^2(t),b_k^2(t)\in \C^{m_d}$
are polynomial of order at most $n_k-1$ for each $k$.

\begin{rem}
  In~\eqref{eq:yrefwdist}, $a_0^1(t)$ and $a_0^2(t)$ correspond to the frequency $\gw_0=0$. The constructions of the controllers are carried out with $\gw_0$ being present, but there are situations where tracking of signals with this frequency component can not be achieved (namely, when the system~\eqref{eq:plantintro} has an invariant zero at $0\in\C$). In this situation the construction of the matrices $G_1$, $G_2$, and $K_1$ in Section~\ref{sec:Controllers} can be modified in a straightforward manner to remove this frequency from the controller.
\end{rem}

Throughout the paper we consider distributed control and observation, i.e.,  $B$ and $C$ are bounded linear operators. Also the disturbance input operator $B_d$ is assumed to be bounded,
but under this assumption it is also possible to reject boundary disturbances
for many classes of PDEs as demonstrated in Section~\ref{sec:1Dheat}. Indeed, since $\wdist(\cdot)$ in~\eqref{eq:yrefwdist} is smooth, boundary disturbances can in many situations be written in the form~\eqref{eq:plantintro} with a bounded operator $B_d$ and a modified disturbance signal including the derivative $\wdistdot(\cdot)$~\citel{CurZwa95}{Sec. 3.3}. 
Since $\wdistdot(\cdot)$ is also of the form~\eqref{eq:yrefwdistwd} with the same frequencies and coefficient polynomial vectors of order at most $n_k-1$, the modified disturbance signal belongs to the same original class of signals.
Moreover, since the operators $B_d$ and $D_d$ are not used in any way in the controller construction in Section~\ref{sec:Controllers}, rejection of boundary disturbances can be done without computing $B_d$ and $D_d$ explicitly --- it is sufficient to know such operators exist.  This extremely useful property is based on the fact that a robust internal model based controller will achieve disturbance rejection for \keyterm{any} disturbance input and feedthrough operators $B_d$ and $D_d$ and any signals of the form~\eqref{eq:yrefwdist}.

The paper is organised as follows. In Section~\ref{sec:RORP} we state the standing assumptions, formulate the robust output regulation problem, and summarise the Galerkin approximations and the balanced truncation method. In Section~\ref{sec:Controllers} we present our main results including the construction of the two finite-dimensional robust controllers. The main theorems are proved in Section~\ref{sec:Proofs}. Section~\ref{sec:PDEcontrol} focuses on robust controller design for particular PDE models. Concluding remarks are presented in Section~\ref{sec:Conclusions}. Section~\ref{sec:Appendix} contains helpful lemmata.  

\subsection{Notation}

The inner product on a Hilbert space $X$ is denoted by $\iprod{\cdot}{\cdot}$.
For a linear operator $A:X\rightarrow Y$ 
we denote by $\Dom(A)$, $\ker(A)$ and $\ran(A)$ the domain, kernel and range of $A$, respectively. The space of bounded linear operators from $X$ to $Y$ is denoted by $\Lin(X,Y)$. If \mbox{$A:X\rightarrow X$,} then $\gs(A)$, $\gs_p(A)$, and $\rho(A)$ denote the spectrum, the point spectrum, and the \mbox{resolvent} set of $A$, respectively.
For $\gl\in\rho(A)$ the resolvent operator is given by \mbox{$R(\gl,A)=(\gl -A)^{-1}$}.
For a fixed $\ga\in\R$ we denote
\begin{align*}
  \mc{H}_\infty(\C_\ga^+) = \bigl\{ G: \C_\ga^+ \to \C \bigm|G \,\text{is analytic,} \sup_{s \in \C_\ga^+} |G(s)| < \infty  \bigr\}
\end{align*}
where $\C_\ga^+=\setm{\gl\in\C}{\re\gl>\ga}$.
For $\ga=0$ we use the notation $\mc{H}_\infty=\mc{H}_\infty(\C_0^+)$.
We denote by $M(\mc{H}_\infty)$
the set of matrices with entries in $\mc{H}_\infty$.

\section{Robust Output Regulation, Galerkin Approximation, and Model Reduction}
\label{sec:RORP}

In this section we state our main assumption on the system~\eqref{eq:plantintro} and the controller and formulate the robust output regulation problem. We also review selected important background results concerning Galerkin approximations and balanced truncation.

We consider a control system~\eqref{eq:plantintro} on a Hilbert space $X$, and we assume $V\subset X$ is another Hilbert space 
with a continuous and dense injection $\iota:V \to X$. 
Let $a(\cdot,\cdot):V\times V\to \C$ be a bounded and coercive sesquilinear form, i.e., there exist $c_1,c_2,\lambda_0 > 0$ such that for all $\phi,~\psi \in V$ we have 
\eq{
  & \quad \abs{a(\phi, \psi)} \le c_1 \norm{\phi}_V \norm{\psi}_V
  \\[1ex]
    & \re a (\phi, \phi) + \lambda_0 \norm{\phi}^2_X   \ge c_2 \norm{\phi}^2_V. 
  }
  We assume $A$ is defined by $a(\cdot,\cdot)$ 
so that
\eq{
  &\iprod{-A\phi}{\psi} = a(\phi,\psi), \qquad 
  \forall
  \phi\in \Dom(A),
  \psi\in V,\\
  &\Dom(A)= \setm{\phi\in V}{a(\phi,\cdot) ~ \mbox{has an extension to} ~ X}.
}
  As shown in~\citel{BanIto97}{Sec. 2}, the operator $A : \Dom(A)\subset X\to X$ is such that   $A-\lambda_0 I $ generates an analytic semigroup on $X$.

  In~\eqref{eq:plantintro} $B$, $C$, and $D$ are the \keyterm{input operator}, \keyterm{output operator} and \keyterm{feedthrough operator}, respectively, and $B_d$ and $D_d$ are the input operator and feedthrough operator, respectively, for the disturbance input $\wdist(t)$.
  These operators are assumed to be bounded
  so that $B\in \Lin(U,X)$, $B_d\in \Lin(U_d,X)$, $C\in \Lin(X,Y)$, $D\in \Lin(U,Y)$, and $D_d\in \Lin(U_d,Y)$ where
  $U=\C^m$ or $U=\R^m$ is the input space,
  $U_d=\C^{m_d}$ or $U_d=\R^{m_d}$ is the disturbance input space,
  and $Y=\C^p$ or $Y=\R^p$ is the output space.
We assume the pair $(A,B)$ is exponentially stablizable and $(C,A)$ is exponentially detectable. 
The transfer function of~\eqref{eq:plantintro} is denoted by
\eq{
  P(\gl) = CR(\gl,A)B + D, \qquad \gl\in\rho(A).
}
We make the following standing assumption which is also necessary for the solvability of the robust output regulation problem. The condition means that $(A,B,C,D)$ is not allowed to have invariant zeros at the frequencies $\set{i\gw_k}_{k=0}^q$ in~\eqref{eq:yrefwdist}.

\begin{ass}
  \label{ass:Psurj}
  Let $K\in \Lin(X,U)$ be such that $A+BK$ generates an exponentially stable semigroup. We assume 
  $P_K(i\gw_k) = (C+DK)R(i\gw_k,A+BK)B+D\in \C^{p\times m}$ is surjective for every $k\in \List[0]{q}$.  
\end{ass}

Due to standard operator identities, the surjectivity of $P_K(i\gw_k)$ is independent of the choice of the stabilizing feedback operator $K$. Moreover, for any $k\in \List[0]{q}$ for which $i\gw_k\in\rho(A)$ the matrix $P_K(i\gw_k)$ is surjective if and only if $P(i\gw_k)$ is surjective.

  We consider the design of internal model based error feedback controllers of the form
  \begin{subequations}
    \label{eq:FeedCon}
    \eqn{
      \dot{z}(t) &= \GG_1 z(t) + \GG_2 e(t)\\
      u(t) &= K z(t)}
  \end{subequations}
  where $e(t) = y(t)- y_{ref}(t)$ is the \keyterm{regulation error}, $\GG_1: \Dom(\GG_1)\subset Z\to Z$ generates a strongly continuous semigroup on $Z$, $\GG_2\in \Lin(Y,Z)$, and $K\in \Lin(Z,U)$.
  Letting $x_e(t)=(x(t),z(t))^T$ and $\wext(t)=(\wdist(t),\yref(t))^T$,
  the system and the controller can be written together as a \keyterm{closed-loop system} on the Hilbert space $X_e=X\times Z$ (see~\cite{HamPoh10,PauPoh10} for details)
    \eq{
      \dot{x}_e(t) &= A_e x_e(t) + B_e \wext(t)
      , \qquad x_e(0)=x_{e0}\\
      e(t) &= C_e x_e(t) + D_e\wext(t)
    }
  where $x_{e0}=(x_0,z_0)^T$ and 
  \eq{
    A_e &= \pmat{A&BK\\\mc{G}_2C&\mc{G}_1+\GG_2 DK}, \quad &B_e &= \pmat{B_d&0\\\mc{G}_2 D_d&-\mc{G}_2}, \\ 
    \quad C_e &= \pmat{C,\;DK}, \quad &D_e &= \pmat{D_d,\;-I}.
  }
  The operator $A_e$ generates a strongly continuous semigroup $T_e(t)$ on $X_e$.

\begin{RORP}
  Choose  $(\mc{G}_1,\mc{G}_2,K)$ in such a way that the following are satisfied:
\begin{itemize}
  \setlength{\itemsep}{.5ex}
  \item[\textup{(a)}] The semigroup $T_e(t)$ is exponentially stable.
  \item[\textup{(b)}] 
    There exists $M_e,\gw_e>0$ such that for all initial states $x_0\in X$ and $z_0\in Z$ and for all signals $\wdist(t)$ and $\yref(t)$ of the form~\eqref{eq:yrefwdist}
we have
\eqn{
\label{eq:errintconv}
\norm{y(t)-\yref(t)}\leq M_e e^{-\gw_e t} (\norm{x_{e0}}+\norm{\Lambda}).
}
  where $\Lambda$ is a vector containing the coefficients of the polynomials $\set{a_k^j(t)}_{k,j}$ and $\set{b_k^j(t)}_{k,j}$ in~\eqref{eq:yrefwdist}.
\item[\textup{(c)}] When
  $(A,B,B_d,C,D,D_d)$ are perturbed to $(\tilde{A},\tilde{B},\tilde{B}_d,\tilde{C},\tilde{D},\tilde{D}_d)$ 
  in such a way that the perturbed closed-loop system remains exponentially stable, 
  then for all $x_0\in X$ and $z_0\in Z$ and for all signals $\wdist(t)$ and $\yref(t)$ of the form~\eqref{eq:yrefwdist}
  the regulation error satisfies~\eqref{eq:errintconv} for some modified constants $\tilde{M}_e,\tilde{\gw}_e>0$.
\end{itemize}
\end{RORP}

  The \keyterm{internal model principle}~\cite[Thm. 6.9]{PauPoh10} implies that in order to achieve robust output tracking of the reference signal $\yref(t)$, it is both necessary and sufficient that the following are satisfied.
\begin{itemize}
  \item The controller~\eqref{eq:FeedCon} incorporates \keyterm{an internal model} of
    the reference and disturbance signals in~\eqref{eq:yrefwdist}.
  \item The semigroup $T_e(t)$ generated by
    $A_e$ is exponentially stable.
\end{itemize}
As shown in Section~\ref{sec:Controllers}, the internal model property of the controller can be guaranteed by choosing a suitable structure for the operator $\GG_1$. The rest of structure and parameters of the controller are then chosen so that the closed-loop system becomes exponentially stable.

\subsection{Background on Galerkin Approximations}
\label{sec:Galbackground}

  Let $V^N$ be a sequence of finite dimensional subspaces of $V $ and
  let $P^N$ be the orthogonal projection of $X$ onto $V^N$.
Throughout the paper we assume the approximating subspaces $(V^N)$ have the property that any element $\phi\in V$ can be approximated by elements in $V^N$ in the norm on $V$, i.e.,
\eqn{
  \label{eq:Galerkinassumption}
  \forall \phi \in V\,\exists (\phi^N)_N, \,\phi^N\in V^N: \quad \norm{\phi^N-\phi}_V \stackrel{N\to\infty}{\longrightarrow} 0 .
}
We define the approximations $A^N:~V^N \to V^N$ of $A$ by 
  \eq{
    \langle -A^N \phi, \psi  \rangle = a (\phi, \psi) \text{~~for all~~} \phi, \psi \in V^N, 
  }
  that is, $A^N$ is defined via the restriction of $a(\cdot,\cdot)$ to $V^N \times V^N$. 
  For $B \in \Lin (U, X)$
  we define $B^N \in \Lin(U, V^N )$ by 
  \eq{
    \langle B^N u, \psi  \rangle = \iprod{u}{ B^* \psi }\text{~~for all~~} 
    \psi \in V^N, 
  }
  and $C^N \in \Lin(V^N, Y)$ is defined as the restriction of $C\in \Lin(X,Y)$ onto $V^N$. 
  Note that computing the Galerkin approximation of $B_d\in \Lin(U_d,X)$ is not necessary.

  \begin{lem}
    \label{lem:GTconv}
    Under the standing assumptions on $A$ and the approximating finite-dimensional subspaces $V^N$, the following hold.
    \begin{itemize}
      \item[\textup{(a)}] If $\tilde{B}\in \Lin(\C^{m_0},X)$ and $\tilde{C}\in \Lin(X,\C^{m_0})$, then
	\eq{
	  \hspace{-3ex}  \norm{P^NR(\gl,A+\tilde{B}\tilde{C})x-R(\gl,A^N+\tilde{B}^N\tilde{C}^N)P^Nx}\stackrel{N\to\infty}{\longrightarrow} 0 
	} 
	for all $\gl\in\rho(A+\tilde{B}\tilde{C})$ and $x\in X$. 
      \item[\textup{(b)}]  
	Let $\tilde{B}\in \Lin(\C^{m_0},X)$ and $\tilde{C}\in \Lin(X,\C^{p_0})$ be such that
	$(A,\tilde{B},\tilde{C})$
	is exponentially stabilizable and detectable.
	If $(\tilde{B}_0^N)_N$ and $(\tilde{C}_0^N)_N$ are two sequences such that 
	$\tilde{B}_0^N\in \Lin(\C^{m_0},V^N)$ and $\tilde{C}_0^N\in \Lin(V^N,\C^{p_0})$ for all $N$ and 
	\eq{
	  \hspace{-2ex}
	  \norm{\tilde{B}_0^N-\tilde{B}}_{\Lin(U,X)}\to 0
	  \quad \mbox{and} \quad
	  \norm{\tilde{C}_0^NP^N-\tilde{C}}_{\Lin(X,Y)}\to 0
	}
	as $N\to \infty$, then $\tilde{C}_0^N R(\cdot,A^N)\tilde{B}_0^N$ converge to the transfer function $\tilde{C}R(\cdot,A)\tilde{B}$ in the graph topology of $M(\Hinf)$ as $N\to \infty$.
    \end{itemize}
  \end{lem}

  \begin{proof}
    It is shown in~\citel{Mor94}{Thm. 5.2} that 
    \eq{
      \norm{P^NR(\gl,A)x-R(\gl,A^N)P^Nx}\stackrel{N\to\infty}{\longrightarrow} 0 \qquad \forall x\in X
    }
    for some $\gl\in\rho(A)$. Since $\tilde{B}^N \tilde{C}^N P^N\to \tilde{B}\tilde{C}$ strongly as $N\to\infty$, the resolvent identity and standard perturbation formulas imply that part (a) is true.

     To prove part (b), 
    let $K\in \Lin(X,U)$ be such that $A+ \tilde{B}K$ is exponentially stable.
  Then by~\citel{Mor94}{Thm. 5.2--5.3} and standard perturbation theory $A^N+\tilde{B}_0^N KP^N$ are uniformly exponentially stable for large $N$. The functions 
$\tilde{C}_0^N R(\cdot,A^N)\tilde{B}_0^N$ and $\tilde{C}R(\cdot,A)\tilde{B}$ have right coprime factorizations in $M(\Hinf)$ given by
\eq{
 & \tilde{C}_0^N R(\cdot,A^N)\tilde{B}_0^N 
   = \tilde{C}_0^N R(\cdot,A^N+\tilde{B}_0^N KP^N)\tilde{B}_0^N \\
   &\hspace{3.3cm} \times(I+KP^N R(\cdot,A^N+\tilde{B}_0^N KP^N)\tilde{B}_0^N)\inv \\
&\tilde{C}R(\cdot,A)\tilde{B} 
  = \tilde{C} R(\cdot,A+\tilde{B} K)\tilde{B} (I+K R(\cdot,A+\tilde{B} K)\tilde{B})\inv.
}
To conclude that $\tilde{C}_0^N R(\cdot,A^N)\tilde{B}_0^N$ converges to $\tilde{C}R(\cdot,A)\tilde{B}$ in the graph topology, it suffices to show that 
$\tilde{C}_0^N R(\cdot,A^N+\tilde{B}_0^N KP^N)\tilde{B}_0^N$ and $KP^N R(\cdot,A^N+\tilde{B}_0^N KP^N)\tilde{B}_0^N$ converge to 
$\tilde{C} R(\cdot,A+\tilde{B} K)\tilde{B}$ and $K R(\cdot,A+\tilde{B} K)\tilde{B}$ in $M(\Hinf)$, respectively. We will only show the convergence of 
$\tilde{C}_0^N R(\cdot,A^N+\tilde{B}_0^N KP^N)\tilde{B}_0^N$ since the second convergence can be shown analogously.

   By~\citel{Mor94}{Thm. 4.2 \& Cor. 4.3} the transfer functions
   $\tilde{C}P^NR(\cdot,A^N+P^N \tilde{B}KP^N)P^N \tilde{B}$ converge to $\tilde{C}R(\cdot,A+\tilde{B}K) \tilde{B}$ in $M(\Hinf(\C_{-\eps}^+))$ for some $\eps>0$.
Standard perturbation theory implies
that for small $\eps>0$ we also have
$\sup_{\re \gl<-\eps}\norm{R(\gl,A^N+\tilde{B}_0^NKP^N)-R(\gl,A^N+P^N\tilde{B}KP^N)}\to 0$ 
as $N\to \infty$.
Together with the convergences of $\tilde{B}_0^N$ and $\tilde{C}_0^N$ and 
the triangle inequality it is easy to show that
$\tilde{C}_0^N R(\cdot,A^N+\tilde{B}_0^N KP^N)\tilde{B}_0^N$ converges to 
$\tilde{C} R(\cdot,A+\tilde{B} K)\tilde{B}$ in $M(\Hinf)$. This completes the proof.
  \end{proof}

\subsection{Model Reduction via Balanced Truncation}
\label{sec:ROM}

  We use balanced truncation~\cite{Moo81, PerSil82} 
  to reduce the order of our controllers.
  For a general minimal and stable finite-dimensional system $(A,B,C)$ on $\C^N$ the reduced order model $(A^r,B^r,C^r)$ on $\C^r$ is computed as follows~\citel{BenFas13}{Sec. 2.1}.
  \begin{itemize}
    \item[(1)] Find a minimal ``internally balanced realization'' $(A_b,B_b,C_b)$
      of  $(A, B,$ $ C)$ as described in~\citel{BenFas13}{Sec. 2.1}.
    \item[(2)]
The controllability Gramian $\Sigma_B\geq 0$ and the observability Gramian $\Sigma_C\geq 0$ of $(A_b,B_b,C_b)$, defined as the solutions of 
      $$
      \begin{cases}
	A_b\Sigma_B+\Sigma_B A_b^\ast=-B_bB_b^\ast\\
	A_b^\ast\Sigma_C+\Sigma_C A_b=-C_b^\ast C_b,
      \end{cases}
      $$
      have the property $\Sigma_B=\Sigma_C=\diag(\sigma_1,\sigma_2,\ldots,\sigma_N)$ where $\sigma_1\geq \sigma_2\geq\cdots\geq \sigma_N>0$ are the \keyterm{Hankel singular values} of $(A,B,C)$.
\item[(3)]
  If we write
  \eq{
    A_b = \pmat{A^r&A_b^{12}\\A_b^{12}&A_b^{22}}, \quad B_b = \pmat{B^r\\B_b^2}, \quad C_b = \pmat{C^r, \; C_b^2}
  }
  where $A^r\in \C^{r\times r}$, $B^r\in \C^{r\times m}$ and $C^r \in\R^{p\times r}$, then $(A^r,B^r,C^r)$ is the desired reduced order model.
  \end{itemize}

\begin{lem}
\label{lem:ROMconv}
The distance in the graph topology between the stable system $(A,B,C)$ on $\C^N$
and its balanced truncation $(A^r,B^r,C^r)$ satisfies 
\eq{
  d(CR(\cdot,A)B,C^rR(\cdot,A^r)B^r)\leq M \sum_{k=r+1}^{N}\sigma_k
}
for some constant $M> 0$ independent of $r\in \List{N}$.
\end{lem}
\begin{proof}
  The convergence in the graph topology follows from the corresponding $M(\Hinf)$-error bound~\cite{Enn84} and the fact that for stable systems the distance in the graph topology and $M(\Hinf)$-norm are equivalent. 
\end{proof}

\begin{rem}
  Improved numerical stability of the model reduction algorithm can be achieved by omitting the explicit computation of the balanced realization
   and instead using a ``balancing-free'' method
  such those in~\cite{Var91} (\texttt{balred} in Matlab)
  or~\cite{SafChi90} (\texttt{hankelmr} in Matlab).
Both of these methods produce reduced order models which satisfy the estimate in Lemma~\ref{lem:ROMconv}.  As demonstrated by the proofs in Section~\ref{sec:Proofs}, the balanced truncation can be replaced by any other model reduction method that 
approximates  
a stable finite-dimensional system in the $M(\Hinf)$-norm.
\end{rem}

\section{Finite-Dimensional Robust Controller Design}
\label{sec:Controllers}

In this section we present algorithms for constructing two finite-dimen\-sional reduced order controllers that solve the robust output regulation problem. 
The constructions use the following data:
\begin{itemize}
  \item Frequencies $\set{\gw_k}_{k=1}^q$ of the reference and disturbance signals~\eqref{eq:yrefwdist}.
  \item Maximal orders $n_k-1$ of the coefficient polynomials 
$a_k^1(t), a_k^2(t), b_k^1(t)$, and $b_k^2(t)$
associated to each $\gw_k$ in~\eqref{eq:yrefwdist}.
\item The dimension of the output space $\dim Y=p$
\item Galerkin approximations $(A^N,B^N,C^N)$ of~\eqref{eq:plantintro}.
\item The values $P(i\gw_k)$ of the transfer function through the invertibility condition of $P(i\gw_k)K_1^{1k}$  (only for the dual observer-based controller when $\dim Y<\dim U$).
\end{itemize}

The construction does not use any information on the disturbance operators $B_d$ and $D_d$ or knowledge of the phases and amplitudes of $\yref(\cdot)$ and $\wdist(\cdot)$.
Indeed, robustness guarantess that the same controller will achieve output tracking and disturbance rejection for \keyterm{any} operators $B_d$ and $D_d$, and for all coefficient polynomials $a_k^1(t), a_k^2(t), b_k^1(t)$, and $b_k^2(t)$ of orders at most $n_k-1$.

In the constructions, the role of the component $G_1$ of the system matrix $\mc{G}_1$ is to guarantee that the controller contains a suitable internal model of the signals~\eqref{eq:yrefwdist}. Expressed in terms of spectral properties, the internal model requires that \textit{$i\gw_k\in\gs_p(\mc{G}_1)$ for all $k\in \List[0]{q}$ and $\mc{G}_1$ has 
at least $p=\dim Y$ independent Jordan chains of length greater than or equal to $n_k$ associated to each eigenvalue $i\gw_k$} (see~\citel{Pau16}{Def. 4}).
The steps following the choice of $G_1$ fix the remaining parameters of the controllers in such a way that the closed-loop system becomes exponentially stable. The choices of the parameters are based on solutions of \textit{finite-dimensional} algebraic Riccati equations involving the Galerkin approximation of~\eqref{eq:plantintro}. 
Increasing the sizes of the parameters $\ga_1,\ga_2\geq 0$ improves the stability margin of the closed-loop system and leads to faster convergence rate for the output, but choosing too large values often causes numerical issues in solving the Riccati equations.
In the final part of the algorithms the order of the finite-dimensional controller is reduced using balanced truncation.

The construction does not give precise bounds for the sizes of the Galerkin approximation or the model reduction, but instead only guarantees that robust output regulation is achieved for approximations of sufficiently high orders. 
As seen in Section~\ref{sec:Proofs}, the key requirement on the orders of these approximations is the ability of the reduced order controller to approximate the behaviour of a full infinite-dimensional observer-based robust controller.
  As Lemma~\ref{lem:ROMconv} indicates, the validity of the reduced order approximation in the graph topology depends on the decay of the Hankel singular values.  
  While for some particular finite-dimensional systems reduction may be impossible (i.e., only the choice $r=N$ is possible for achieving a given accuracy), the Hankel singular values of Galerkin approximations of parabolic PDE systems typically decay fairly rapidly and because of this reduction is usually possible.

The main results, Theorems~\ref{thm:OBScontr} and~\ref{thm:DOBScontr}, confirm that the constructed controllers solve the robust output regulation problem. 
The proofs of the theorems are presented in Section~\ref{sec:Proofs}.
The proofs also show that the Riccati equations in Step~3 can be solved approximately in order to improve computational efficiency,  as long as the approximation scheme is such that the approximation errors of $K^N$ and $L^N$ are small.

\subsection{Observer-Based Finite-Dimensional Controller}
\label{sec:OBScontr}

Our first finite-dimensional robust controller is of the form
\begin{subequations}
  \label{eq:FinConObs}
  \eqn{
    \dot{z}_1(t)&= G_1z_1(t) + G_2 e(t)\\
  \label{eq:FinConObs2}
    \dot{z}_2(t)&= (A_L^r+B_L^rK_2^r)z_2(t) + B_L^r K_1^N z_1(t) -L^r e(t)\\
    u(t)&= K_1^N z_1(t) + K_2^rz_2(t) 
  }
\end{subequations}
with state $(z_1(t),z_2(t))^T\in Z:= Z_0\times \C^r$ and input $e(t)=y(t)-\yref(t)$.
  The matrices $(G_1,G_2,A_L^r,B_L^r,K_1^N,K_2^r,L^r)$ are chosen using the algorithm below.
  More precisely,
$(G_1,G_2)$ are as in Step~1, $K_1^N$ is as in Step~3, and $(A_L^r,B_L^r,L^r,K_2^N)$ are as in Step~4.  
The parts $G_1,G_2,K_1^N$ are \keyterm{the internal model} in the controller.
The terminology ``observer-based controller'' arises from the property that the finite-dimensional subsystem~\eqref{eq:FinConObs2} 
approximates 
(in a certain sense)
a full infinite-dimensional observer for~\eqref{eq:plantintro}.

\smallskip

\noindent\textbf{PART I. The Internal Model}

\smallskip

\noindent \textbf{Step 1:}
  We choose $Z_0=Y^{n_0}\times Y^{2n_1} \times \ldots \times Y^{2n_q}$, $G_1 = \diag(J_0^Y, \ldots, J^Y_q)\in \Lin(Z_0)$, and 
  $G_2=(G_2^k)_{k=0}^q \in \Lin(Y,Z_0)$. 
  The parts of $G_1$ and $G_2$ are chosen as follows.
  For $k=0$, let
\begin{align*}
J^Y_0 = \pmat{
0_p  & I_p  	&  			&  \\
      	&  0_p & \ddots	&  \\
      	&			& \ddots	& I_p    \\
      	&			& 			& 0_p
} 
, \qquad 
G_2^0 = \pmat{0_p\\\vdots\\0_p\\I_p}
\end{align*}
where $0_p $ and $I_p$ are the $p\times p$ zero and identity matrices, respectively. 
  For $k \in \{ 1, \ldots, q \}$ we choose 
\begin{align*}
J^Y_k = \pmat{
  \Omega_k   & I_{2p}  &  &  \\
  &  \Omega_k  & \ddots&  \\
  && \ddots& I_{2p}    \\
  && & \Omega_k 
}
, \qquad 
G_2^k = \pmat{0_{2p}\\\vdots\\0_{2p}\\I_p\\0_p}
\end{align*}
where $\Omega_k = \pmatsmall{0_p&\gw_k I_p\\-\gw_k I_p&0_p}$.
The pair $(G_1,G_2)$ is controllable by construction.

\smallskip

\noindent \textbf{PART II. The Galerkin Approximation and Stabilization}.

\smallskip

\noindent \textbf{Step 2:}
  For a fixed and sufficiently large $N\in\N$, apply the Galerkin approximation described in Section~\ref{sec:Galbackground} to the system $(A,B,C)$ to arrive at the finite-dimensional system $(A^N,B^N,C^N)$ on $V^N$.

\smallskip
\noindent \textbf{Step 3:}
Choose the parameters $\ga_1,\ga_2\geq 0$, $Q_1\in \Lin(U_0,X)$, and $Q_2\in \Lin(X,Y_0)$ with $U_0,Y_0$ Hilbert in such a way that the systems
$(A + \ga_1 I,Q_1,C)$  and $(A+\ga_2 I,B,Q_2)$
are both exponentially stabilizable and detectable.
Let $Q_1^N$ and $Q_2^N$ be the approximations of $Q_1$ and $Q_2$, respectively, according to the approximation $V^N$ of $V$.
Let $Q_0\in \Lin(Z_0,\C^{p_0})$ be such that $(Q_0,G_1)$ is observable, and let $R_1\in \Lin(Y)$ and $R_2\in \Lin(U)$ be positive definite matrices.
Denote
\eq{
  \Acomp^N = \pmat{G_1&G_2C^N\\0&A^N}, \quad \Bcomp^N=\pmat{G_2D\\B^N}, \quad \Qcomp^N = \pmat{Q_0&0\\0&Q_2^N}.
}
Define
$L^N =-\Sigma_N C^N R_1\inv\in \Lin(Y, V^N) $ 
and define 
$K^N = \pmat{K_1^N,\; K_2^N} =-R_2\inv (\Bcomp^N)^\ast\Pi_N \in \Lin(Z_0\times V^N ,U )$ 
where
$\Sigma_N$
and
$\Pi_N$
are the non-negative solutions of the finite-dimensional Riccati equations
\eq{
  (A^N + \ga_1 I) \Sigma_N + \Sigma_N (A^N + \ga_1 I)^* 
  - \Sigma_N \left(C^N \right)^* R_1\inv C^N \Sigma_N &=- Q_1^N (Q_1^N)^*  \\
  (\Acomp^N + \ga_2 I)^* \Pi_N + \Pi_N (\Acomp^N +\ga_2 I) 
  - \Pi_N \Bcomp^N R_2\inv\left(\Bcomp^N\right)^\ast \Pi_N &=- 
  \left(\Qcomp^N\right)^\ast \Qcomp^N.  
}
The exponential stabilizability of the pair $(\Acomp^N+\ga_2 I,\Bcomp^N)$ for large $N$ follows from~\citel{Mor94}{Sec. 5.2} and Lemma~\ref{lem:BlockOpStabDet}.
With the above choices the matrices $\Acomp^N+\Bcomp^NK^N$ and $A^N+L^NC^N$ are Hurwitz if $N$ is sufficiently large~\citel{BanIto97}{Thm. 4.8}.

\smallskip

\noindent\textbf{PART III.} \textbf{The Model Reduction} 

\smallskip

\noindent \textbf{Step 4:}
  For a fixed and suitably large $r\in\N$, $r\leq N$,
  apply the balanced truncation method in Section~\ref{sec:ROM} to the stable finite-dimensional system
  \eq{
    (A^N+L^NC^N, [ B^N+L^N D,\;L^N],K^N_2)
  }
to obtain a stable $r$-dimensional reduced order system
\eq{
  \left(A_L^r,[B_L^r ,\; L^r] ,K_2^r\right).
}

\begin{thm} \label{thm:OBScontr}
  Let Assumption~\textup{\ref{ass:Psurj}} be satisfied.
The finite-dimensional controller~\eqref{eq:FinConObs} solves the Robust Output Regulation Problem provided that the order $N$ of the Galerkin approximation and the order $r$ of the model reduction are sufficiently high.

If $\ga_1,\ga_2>0$, then the controller achieves a uniform stability margin in the sense that for any fixed $0<\ga<\min \set{\ga_1,\ga_2}$ the 
 operator $A_e+\ga I$ will generate an exponentially stable semigroup 
if $N$ and $r\leq N$ are sufficiently large.  
\end{thm}

\subsection{Dual Observer-Based Finite-Dimensional Controller}
\label{sec:DOBScontr}

The second controller we construct is of the form
\begin{subequations}
  \label{eq:FinConDObs}
  \eqn{
    \dot{z}_1(t)&= G_1z_1(t) + G_2^N C_K^r z_2(t) + G_2^N e(t)\\
    \dot{z}_2(t)&= (A_K^r+L^rC_K^r)z_2(t)  +L^r e(t)\\
    u(t)&= K_1z_1(t) - K_2^rz_2(t) 
  }
\end{subequations}
  with state $(z_1(t),z_2(t))\in Z:= Z_0\times \C^r$,
  and the matrices $(G_1,G_2^N,A_K^r,C_K^r,$ $K_1,K_2^r,L^r)$ are chosen using the algorithm below.
More precisely,
$(G_1,K_1)$ are as in Step~1, $G_2^N$ is as in Step~3, and $(A_K^r,C_K^r,K_2^r,L^r)$ are as in Step~4.
  The terminology ``dual observer-based controller'' is motivated by the property that the dual system of~\eqref{eq:FinConDObs} will in fact achieve closed-loop stability with the dual $(A^\ast,C^\ast,B^\ast,D^\ast)$ of the original system~\eqref{eq:plantintro}. Since $X_e$ is a Hilbert space, we can use this property in proving closed-loop stability in Section~\ref{sec:Proofs}.

\smallskip

\noindent\textbf{PART I. The Internal Model}

\smallskip

\noindent \textbf{Step 1:}
  We choose $Z_0=Y^{n_0}\times Y^{2n_1} \times \ldots \times Y^{2n_q}$, $G_1 = \diag(J_0^Y, \ldots, J^Y_q)\in \Lin(Z_0)$, and 
  $K_1=[K_1^0,\, \ldots, \, K_1^q] \in \Lin(Z_0,U)$. 
  The parts of $G_1$ and $K_1$ are chosen as follows.
  For $k=0$, let
\begin{align*}
J^Y_0 = \pmat{
0_p  & I_p  	&  			&  \\
      	&  0_p & \ddots	&  \\
      	&			& \ddots	& I_p    \\
      	&			& 			& 0_p
} 
\end{align*}
and $K_1^0=[K_1^{01},\;0_p,\;\ldots,\; 0_p]$,
where $0_p $ and $I_p$ are the $p\times p$ zero and identity matrices, respectively. 
  For $k \in \{ 1, \ldots, q \}$ we choose 
\begin{align*}
J^Y_k = \pmat{
  \Omega_k   & I_{2p}  &  &  \\
  &  \Omega_k  & \ddots&  \\
  && \ddots& I_{2p}    \\
  && & \Omega_k 
}
, \qquad \Omega_k = \pmat{0_p&\gw_k I_p\\-\gw_k I_p&0_p}
\end{align*}
and $K_1^k = [K_1^{k1},0_p,0_{2p},\ldots,0_{2p}]$.
For each $k\in \List[0]{q}$ the matrices $K_1^{k1}\in \Lin(Y,U)$ are chosen\footnote{This choice is possible by Assumption~\ref{ass:Psurj} whenever $i\gw_k\in\rho(A)$.
  If $i\gw_k\notin \rho(A)$ for some $k$, then we instead choose $K_1^{1k}$ in such a way that
  $P_L(iw_k) K_1^{k 1} \in \Lin(Y)$ is boundedly invertible where $P_L(\gl)=CR(\gl,A+LC)(B+LD)+D$ with some $L\in \Lin(Y,X)$ such that $A+LC$ is exponentially stable. The invertibility of $P_L(iw_k) K_1^{k 1} \in \Lin(Y)$  does not depend on the choice of $L$ due to the identity $P_{\tilde{L}}(i\gw_k)=(I-CR(i\gw_k,A+LC)(\tilde{L}-L))\inv P_L(i\gw_k)$ where $\tilde{L}\in \Lin(Y,X)$ is another operator for which $A+\tilde{L}C$ is exponentially stable.}
so that $P(i\gw_k)K_1^{k1}\in \Lin(Y)$ are boundedly invertible for all $k\in \List[0]{q}$.
If $m=p$, we can choose $K_1^{k1}=I_p$ for all $k\in \List[0]{q}$.
    The pair $(K_1,G_1)$ is observable by construction.

\smallskip

\noindent \textbf{PART II. The Galerkin Approximation and Stabilization}.

\smallskip

\noindent \textbf{Step 2:}
  For a fixed and sufficiently large $N\in\N$, apply the Galerkin approximation described in Section~\ref{sec:Galbackground} to the system $(A,B,C)$ to arrive at the finite-dimensional system $(A^N,B^N,C^N)$ on $V^N$.

\smallskip

\noindent \textbf{Step 3:}
Choose the parameters $\ga_1,\ga_2\geq 0$, $Q_1\in \Lin(X,Y_0)$, and $Q_2\in \Lin(U_0,X)$ with $U_0,Y_0$ Hilbert in such a way that the systems
$(A + \ga_1 I,B,Q_1)$  and $(A+\ga_2 I,Q_2,C)$
are both exponentially stabilizable and detectable.
Let $Q_1^N$ and $Q_2^N$ be the approximations of $Q_1$ and $Q_2$, respectively, according to the approximation $V^N$ of $V$.
Let $Q_0\in \Lin(\C^{p_0},Z_0)$ be such that $(G_1,Q_0)$ is controllable, and $R_1\in \Lin(U)$ and $R_2\in \Lin(Y)$ be positive definite matrices.
Denote
$ \Ccomp^N=\pmat{DK_1,\; C^N}$ and
\eq{
  \Acomp^N = \pmat{G_1&0\\B^NK_1&A^N},
  \quad \Qcomp^N = \pmat{Q_0&0\\0&Q_2^N},
}
Define
$\mc{G}_2^N:= \pmatsmall{G_2^N\\L^N} = -\Pi_N C_s^N R_2\inv\in \Lin(Y,Z_0\times V^N)$
and define
$K_2^N = -R_1\inv (B^N)^\ast \Sigma_N\in \Lin(V^N,U)$ 
where
$\Sigma_N$
and
$\Pi_N$
are the non-negative solutions of the finite-dimensional Riccati equations
\eq{
  (A^N + \ga_1 I)^\ast \Sigma_N + \Sigma_N (A^N + \ga_1 I)  
   - \Sigma_N B^N R_1\inv \left(B^N\right)^\ast \Sigma_N &=- (Q_1^N)^* Q_1^N  \\
  (\Acomp^N + \ga_2 I) \Pi_N + \Pi_N (\Acomp^N +\ga_2 I)^\ast 
   - \Pi_N \left(\Ccomp^N\right)^\ast R_2\inv\Ccomp^N \Pi_N &=- 
  \Qcomp^N \left(\Qcomp^N\right)^\ast .  
}
The exponential detectability of the pair $(\Ccomp^N,\Acomp^N+\ga_2I)$ for large $N$ follows from~\citel{Mor94}{Sec. 5.2} and Lemma~\ref{lem:BlockOpStabDet}.
With these choices the matrices $A^N+B^NK_2^N$ and $\Acomp^N+\mc{G}_2^N\Ccomp^N$ are Hurwitz if $N$ is sufficiently large~\citel{BanIto97}{Thm. 4.8}.

\smallskip

\noindent\textbf{PART III.} \textbf{The Model Reduction} 

\smallskip

\noindent \textbf{Step 4:}
  For a fixed and suitably large $r\in\N$, $r\leq N$
  apply the balanced truncation method in Section~\ref{sec:ROM} to the stable finite-dimensional system
\begin{align*}
  \biggl(A^N+B^NK^N_2, L^N,
\begin{bmatrix}
C^N+DK_2^N \\ K^N_2
\end{bmatrix}
\biggr)
\end{align*}
to obtain a stable $r$-dimensional reduced order system
\eq{
  \biggl(
  A_K^r,L^r,
  \begin{bmatrix} C_K^r\\ K_2^r \end{bmatrix}
  \biggr)
.
}

\begin{thm} \label{thm:DOBScontr}
  Let Assumption~\textup{\ref{ass:Psurj}} be satisfied.
The finite-dimensional controller~\eqref{eq:FinConDObs} solves the Robust Output Regulation Problem provided that the order $N$ of the Galerkin approximation and the order $r$ of the model reduction are sufficiently high.

If $\ga_1,\ga_2>0$, then the controller achieves a uniform stability margin in the sense that for any fixed $0<\ga<\min \set{\ga_1,\ga_2}$ the 
 operator $A_e+\ga I$ will generate an exponentially stable semigroup 
if $N$ and $r\leq N$ are sufficiently large.  
\end{thm}

\section{Proofs of the Main Results}
\label{sec:Proofs}

The proofs of Theorems~\ref{thm:OBScontr} and \ref{thm:DOBScontr} are based on the internal model principle which states that a controller solves the robust output regulation problem provided that it contains an internal model of the frequencies of $\yref(t)$ and $\wdist(t)$ and the closed-loop system is exponentially stable.

  In showing the closed-loop stability we employ a combination of perturbation and approximation arguments. We first construct an infinite-dimensional controller $(\GG_1^\infty,\GG_2^\infty,K^\infty)$ which stabilizes the closed-loop system
 and then compare the distance between two closed-loop systems --- one with our controller $(\GG_1,\GG_2,K)$ and one with $(\GG_1^\infty,\GG_2^\infty,K^\infty)$ --- in the graph topology for large $N$ and $r$.
To ensure the stabilizability and detectability of the closed-loop systems, we consider them with suitable modified input and output operators $\tilde{B}_e$ and $\tilde{C}_e$.
We then prove that $(A_e,\tilde{B}_e,\tilde{C}_e)$ is input-output stable
by showing that for sufficiently large $N$ and $r$
the distance of this system in the graph topology 
to the input-output stable closed-loop system $(A_e^\infty,\tilde{B}_e^\infty,\tilde{C}_e^\infty)$
can be made arbitrarily small.
The input-output stability together with stabilizability and detectability of $(A_e,\tilde{B}_e,\tilde{C}_e)$ will finally imply that $T_e(t)$ is exponentially stable.

In summary, the proof consists of the following parts:
\begin{itemize}
  \item[1.] Verify that $(\GG_1,\GG_2,K)$   has an internal model.
  \item[2.] Define an exponentially stabilizable and detectable closed-loop system $(A_e,\tilde{B}_e,\tilde{C}_e)$ with suitable 
    $\tilde{B}_e$ and $\tilde{C}_e$. 
    The input-output stability of this system will imply the exponential stability of $T_e(t)$
  by~\citel{Reb93}{Cor. 1.8}.
  \item[3.] Construct a stabilizing infinite-dimensional controller $(\GG_1^\infty,\GG_2^\infty,K^\infty)$ and the corresponding input-output stable closed-loop system $(A_e^\infty,$ $\tilde{B}_e^\infty,\tilde{C}_e^\infty)$.
  \item[4.] Show that for large $N$ and $r$ the distance in graph topology between $(A_e,\tilde{B}_e,\tilde{C}_e)$ and $(A_e^\infty,\tilde{B}_e^\infty,\tilde{C}_e^\infty)$ becomes arbitrarily small, and thus
$(A_e,\tilde{B}_e,\tilde{C}_e)$ is input-output stable 
    for sufficiently large $N$ and $r$~\cite{JacNet88,Mor94}.
  \item[5.] Combine parts~1, 2, and~4 to conclude that  $(\mc{G}_1,\mc{G}_2,K)$ solves the robust output regulation problem.
\end{itemize}

\begin{proof}[Proof of Theorem~\textup{\ref{thm:OBScontr}}]
  The matrices $(\mc{G}_1,\mc{G}_2,K)$ of the error feedback controller~\eqref{eq:FeedCon} are given by
  \eq{
    \GG_1 = 
    \pmat{
      G_1 & 0 \\ B_L^r K_1^N & A_L^r+ B_L^r K_2^r
    }
    , \qquad
    \GG_2  =
    \pmat{
      G_2  \\ -L^r  
    }
    , 
  }
    $K  = \pmat{K_1^N,\, K_2^r }$,
and $Z=Z_0\times \C^r$ or $Z=Z_0\times \R^r$.
  If $\ga_1>0$ and $\ga_2>0$ we let $0<\ga <\min \set{\ga_1,\ga_2}$ be arbitrary. Otherwise we take $\ga=0$.

\noindent \textbf{Part 1 -- The Internal Model Property:}
The block structures of $\GG_1$ and $\GG_2$ are the same as in the controller constructed in~\citel{Pau16}{Sec. VI}. 
The matrices $G_1$ and $G_2$ are related to the corresponding matrices in~\citel{Pau16}{Sec. VI} through a similarity transform. Since the internal model property is invariant under such transformations,
the argument at the end of the proof of~\citel{Pau16}{Thm. 15} shows that if the closed-loop is exponentially stable, then the controller $(\GG_1,\GG_2,K)$ has an internal model in the sense that (see~\citel{Pau16}{Def. 5})
    \eq{
      \ran(i\gw_k-\GG_1)\cap \ran(\GG_2)&=\set{0}, \qquad \qquad 0\leq k\leq q\\
      \ker(\GG_2)&=\set{0}\\
      \ker(i\gw_k-\GG_1)^{n_k-1}&\subset \ran(i\gw_k-\GG_1) \quad 0\leq k\leq q.
    }

\medskip

\noindent\textbf{Part 2 -- A Modified Closed-Loop System:} 
Consider a composite system 
$(A_{e0},\tilde{B}_e,\tilde{C}_e)$
with
\eq{
  A_{e0}&= \pmat{A&0\\0&\GG_1}, \quad\quad\; \tilde{B}_e=
  \pmat{B&0\\0&\tilde{\GG}_2},
  \qquad \tilde{C}_e= \pmat{C&0\\0&\tilde{K}} ,\\
  \quad \tilde{\GG}_2&=\pmat{G_2&0\\-L^r&B_L^r},  
  \quad \tilde{K}=\pmat{K_1^N&0\\0&K_2^r}.
}
If $N$ is large, then $A^N+\ga I+L^NC^N$ is exponentially stable by~\citel{BanIto97}{Thm. 4.8}. 
Since $A_L^r$ is obtained from $A^N+L^N C^N$ using balanced truncation,
also $A_L^r+\ga I$ in $\mc{G}_1$ is Hurwitz for large $N$ and $r$.
The pair $(G_1+\ga I,G_2)$ is controllable 
by construction, and $(K_1^N,G_1+\ga I)$ is observable by Lemma~\ref{lem:BlockOpStabDet}.
Using these properties it is easy to see that $(\GG_1 + \ga I,\tilde{\GG}_2,\tilde{K})$  is exponentially stabilizable and detectable for large $N$ and $r$, and therefore the same holds for $(A_{e0} + \ga I,\tilde{B}_e,\tilde{C}_e)$. 
A direct computation shows that $A_e=A_{e0} +\tilde{B}_eK_e \tilde{C}_e$ where
\eqn{
  \label{eq:Kefeedback}
  K_e=\pmat{0&I&I\\I&D&D\\0&0&0} 
}
and thus under the output feedback with the operator $K_e$ the system $(A_{e0}+\ga I,\tilde{B}_e,\tilde{C}_e)$ becomes $(A_e+\ga I,\tilde{B}_e,\tilde{C}_e)$.
Since output feedback preserves stabilizability and detectability, for large $N$ and $r\leq N$ the input-output stability of $(A_e+\ga I,\tilde{B}_e,\tilde{C}_e)$ will imply the exponential stability of the semigroup $e^{\ga t}T_e(t)$ generated by $A_e+\ga I$~\citel{Reb93}{Cor. 1.8}.

\medskip

\noindent \textbf{Part 3 -- An Infinite-Dimensional Stabilizing Controller $(\GG_1^\infty,\GG_2^\infty,$ $K^\infty)$:}
Choose $Z_\infty = Z_0\times X$ and 
\eq{
  \GG_1^\infty = 
  \pmat{
    G_1 & 0 \\ (B+L^\infty D)K_1^\infty  & A+L^\infty C+ (B+L^\infty D)K_2^\infty
  }
  , 
}
  and $\GG_2^\infty  = \pmatsmall{ G_2  \\ -L^\infty  }$
where $K^\infty:=[K_1^\infty,K_2^\infty]$ and $L^\infty$ are the limits of $K^N$ and $L^N$ in the sense that
\eq{
  \norm{K^N \bmatsmall{I&0\\0&P^N} -K^\infty}_{\Lin(Z_0\times X,U)}\to 0 
  \quad \mbox{and} \quad
  \norm{P^NL^N-L^\infty}_{\Lin(Y,X)}\to 0
}
as $N\to \infty$.
Here $P^N:X\to V^N$ is again the Galerkin projection onto $V^N$. The limit $L^\infty$ exists due to the approximation theory for solutions of Riccati operator equations~\citel{BanIto97}{Thm. 4.8}.
Moreover, if we define 
\eq{
  \Acomp = \pmat{G_1&G_2C\\0&A}, \qquad \Bcomp=\pmat{G_2D\\B}, \qquad \Qcomp = \pmat{Q_0&0\\0&Q_2}
}
then it is straightforward to show based on properties of $A$ that the form defined by $a_s(\phi,\psi)=\iprod{-\Acomp\phi}{\psi}$, $\phi\in \Dom(\Acomp)$, $\psi\in Z_0\times X$ and the approximating subspaces $V_s^N=Z_0\times V^N$ satisfy the assumptions
of~\citel{BanIto97}{Thm. 4.8}. 
Since $(\Acomp+\ga_2 I,\Bcomp, \Qcomp)$ is exponentially stabilizable and detectable by Lemma~\ref{lem:BlockOpStabDet}, also the existence of $K^\infty$ follows from~\citel{BanIto97}{Thm. 4.8}.
Moreover, the semigroups generated by $A+\ga I+L^\infty C$ and $\Acomp+\ga I+\Bcomp K^\infty$ are exponentially stable.

We will now show that $A_e^\infty $ ---  the closed-loop system operator with $(\mc{G}_1,\mc{G}_2,K)$ replaced by $(\GG_1^\infty,\GG_2^\infty,K^\infty)$ --- is such that $A_e^\infty +\ga I$ generates an exponentially stable semigroup.
If we define a bounded similarity transform
\eq{
  \Lambda_e = \pmat{0&I&0\\I&0&0\\-I&0&I}, \qquad
  \Lambda_e\inv = \pmat{0&I&0\\I&0&0\\0&I&I},
}
then a direct computation shows that
\eq{
\MoveEqLeft[0]\Lambda_eA_e^\infty \Lambda_e\inv=\Lambda_e\pmat{A&BK^\infty\\\mc{G}_2^\infty C&\mc{G}_1^\infty+\GG_2^\infty DK^\infty}\Lambda_e\inv\\
&= \Lambda_e\hspace{-.4ex}\pmat{A&BK_1^\infty&BK_2^\infty\\G_2C&G_1+G_2DK_1^\infty&G_2DK_2^\infty\\-L^\infty C&BK_1^\infty&A+BK_2^\infty+L^\infty C}\hspace{-.4ex}\Lambda_e\inv\\
&=\pmat{G_1+G_2DK_1^\infty&G_2(C+DK_2^\infty)&G_2DK_2^\infty\\BK_1^\infty&A+BK_2^\infty&BK_2^\infty\\0&0&A+L^\infty C}.
}
The first $2\times 2$ subsystem of $\Lambda_eA_e^\infty \Lambda_e\inv$ is given by  
\eq{
  \pmat{G_1&G_2C\\0&A} + \pmat{G_2D\\B} \pmat{K_1^\infty,\; K_2^\infty} 
  = \Acomp+\Bcomp K^\infty.
  }
  Since $A+\ga I+L^\infty C$ and $\Acomp + \ga I+\Bcomp K^\infty$ generate exponentially stable semigroups, the same is true for  $\Lambda_e(A_e^\infty + \ga I) \Lambda_e\inv$ and $A_e^\infty + \ga I$.

Finally, define $(A_{e0}^\infty,\tilde{B}_e^\infty,\tilde{C}_e^\infty)$ 
\eq{
  A_{e0}^\infty= \pmat{A&0\\0&\GG_1^\infty}, \qquad \tilde{B}_e=
  \pmat{B&0\\0&\tilde{\GG}_2^\infty},
 \qquad \tilde{C}_e= \pmat{C&0\\0&\tilde{K}^\infty} 
}
where 
\eq{
  \tilde{\GG}_2^\infty=\pmat{G_2&0\\-L^\infty &B+L^\infty D} \quad \mbox{and} \quad \tilde{K}^\infty=\pmat{K_1^\infty&0\\0&K_2^\infty}.
}
Output feedback with the feedback operator 
in~\eqref{eq:Kefeedback} 
transforms $(A_{e0}^\infty + \ga I,\tilde{B}_e^\infty,$ $\tilde{C}_e^\infty)$ to  $(A_e^\infty + \ga I,\tilde{B}_e^\infty,\tilde{C}_e^\infty)$. The system $(A_e^\infty + \ga I,\tilde{B}_e^\infty,\tilde{C}_e^\infty)$ is input-output stable since $A_e^\infty + \ga I$ generates an exponentially stable semigroup.

\medskip

\noindent \textbf{Part 4 -- Input-Output Stability of $(A_e,\tilde{B}_e,\tilde{C}_e)$:}
Our aim is to show that for large $N$ and $r$ the distance in graph topology between 
$(A_e+\ga I,\tilde{B}_e,\tilde{C}_e)$
and $(A_e^\infty + \ga I,\tilde{B}_e^\infty,\tilde{C}_e^\infty)$
can be made arbitrarily small.
By Lemma~\ref{lem:CLvsOLconv} and Part~3
 it is sufficient to show that the distance between $(A_{e0}^\infty + \ga I,\tilde{B}_e^\infty,\tilde{C}_e^\infty)$
 and $(A_{e0}+\ga I,\tilde{B}_e,\tilde{C}_e)$ becomes small for large $N$ and $r$. Due to the structure of these systems this is true if (and only if) the distance in graph topology between  $(\GG_1 + \ga I,\tilde{\GG}_2,\tilde{K})$ and $(\GG_1^\infty + \ga I,\tilde{\GG}_2^\infty,\tilde{K}^\infty)$ becomes small. 
If we define
\eq{
  \mc{G}_{10} = \pmat{G_1&0\\0&A_L^r}, \, 
  \mc{G}_{10}^\infty = \pmat{G_1&0\\0&A+L^\infty C}, \, K_c = \pmat{0&0\\I&I}
}
we see that $\GG_1=\GG_{10}+\tilde{\GG}_2K_c \tilde{K}$ and $\GG_1^\infty=\GG_{10}^\infty+\tilde{\GG}_2^\infty K_c \tilde{K}^\infty$. Therefore Lemma~\ref{lem:CLvsOLconv} and the structure of the controllers imply that the distance between $(\GG_1 + \ga I,\tilde{\GG}_2,\tilde{K})$ and $(\GG_1^\infty + \ga I,\tilde{\GG}_2^\infty,\tilde{K}^\infty)$
can be made small provided that the distance $d(\mc{P},\mc{P}_r)$
in the graph topology between 
\eq{
  \mc{P}&:= (A + \ga I+L^\infty C,[ B+L^\infty D,\;L^\infty ],K_2^\infty) \quad \mbox{and} \\
  \mc{P}_r&:=(A_L^r + \ga I,[B_L^r,\; L^r],K_2^r)
}
becomes arbitrarily small for large $r$ and $N$.
The triangle inequality implies 
$d(\mc{P},\mc{P}_r)\leq d(\mc{P},\mc{P}_N)+d(\mc{P}_N,\mc{P}_r)$ where
$\mc{P}_N:=(A^N + \ga I+L^NC^N,[B^N+L^ND,\, L^N],K_2^N)$.
Since $\mc{P}$ and $\mc{P}_N$ are parts of systems obtained with output feedback from $(A + \ga I,[ B+L^\infty D,\;L^\infty ],\pmatsmall{C&0\\0&K_2^\infty})$ and $(A^N + \ga I,[B^N+L^ND,\; L^N],\pmatsmall{C^N&0\\0&K_2^N})$, respectively, Lemmas~\ref{lem:CLvsOLconv} and~\ref{lem:GTconv} imply 
$d(\mc{P},\mc{P}_N)\to 0$ as $N\to\infty$.  
Finally, since
$\mc{P}_r$ is the system obtained from $\mc{P}_N$ using model reduction, we have from Lemma~\ref{lem:ROMconv} that $d(\mc{P}_N,\mc{P}_r)$ can be made arbitrarily small by choosing a sufficiently large $r\leq N$ (in the extreme case only the choice $r=N$ may be possible, in which case $d(\mc{P}_N,\mc{P}_r)=0$).

\medskip

\noindent \textbf{Part 5 -- Conclusion:}
By Part~1 the controller contains an internal model and by Parts~2--4 the semigroup $e^{\ga t}T_e(t)$ generated by $A_e+\ga I$ is exponentially stable. We have from~\citel{Pau16}{Thm. 7} that the controller solves robust output regulation problem\footnote{In the reference~\cite{Pau16} the objective of the robust output regulation problem was to achieve $t\mapsto e^{\ga t}\norm{e(t)}\in \Lp[2](0,\infty;Y)$ for some $\ga>0$, but since in our case $B$, $C$, $\mc{G}_2$ and $K$ are bounded operators, the expression for $e(t)$ in the proof of~\citel{Pau16}{Thm. 7} implies that also~\eqref{eq:errintconv} is satisfied.}.
\end{proof}

\begin{proof}[Proof of Theorem~\textup{\ref{thm:DOBScontr}}]
  The matrices $(\mc{G}_1,\mc{G}_2,K)$ of the error feedback controller~\eqref{eq:FeedCon} are given by
\eq{
  \GG_1 = 
  \pmat{
    G_1 & G_2^N  C_K^r \\ 0  & A_K^r+L^r C_K^r
  }
  , \qquad
  \GG_2  =
  \pmat{
    G_2^N  \\ L^r  
  }
  ,
}
$ K  = \pmat{K_1,\, -K_2^r }$,
and $Z=Z_0\times \C^r$ or $Z=Z_0\times \R^r$.

\noindent \textbf{Part 1 -- The Internal Model Property:}
Due to the properties of $G_1$ and the block structure of $\mc{G}_1$, the controller contains an internal model of the reference and disturbance signals in the sense that
$\dim \ker(i\gw_k-\mc{G}_1)\geq \dim Y=p$
for all $k\in\List[0]{q}$
and $\mc{G}_1$ has at least $p$ independent Jordan chains of length greater than or equal to $n_k$ associated to each eigenvalue $i\gw_k$ (see~\citel{Pau16}{Def. 4}).

\noindent \textbf{Part 2 -- Stability of the Closed-Loop System:}
  If $\ga_1>0$ and $\ga_2>0$ we let $0<\ga <\min \set{\ga_1,\ga_2}$ be arbitrary. Otherwise we take $\ga=0$.
We will prove exponential closed-loop stability by showing that the adjoint $A_e^\ast + \ga I$ of $A_e +\ga I$ generates an exponentially stable semigroup. 
The adjoint operator $A_e^\ast$ is given by
\eq{
  A_e^\ast = \pmat{A^\ast &C^\ast \mc{G}_2^\ast \\K^\ast B^\ast & \mc{G}_1^\ast + K^\ast D^\ast \mc{G}_2^\ast}
}
where
  $\GG_2^\ast  =
  \pmat{ (G_2^N)^\ast, \, (L^r)^\ast }$,  
  $K^\ast  = \pmatsmall{K_1^\ast\\ -(K_2^r)^\ast }$,
\eq{
  \GG_1^\ast = 
  \pmat{
    G_1^\ast & 0\\  (C_K^r)^\ast (G_2^N)^\ast    & (A_K^r)^\ast+ (C_K^r)^\ast (L^r)^\ast
  }
.
}
  The dual  $(\mc{G}_1^\ast,K^\ast, \mc{G}_2^\ast)$ of $(\mc{G}_1,\mc{G}_2,K)$ coincides with a controller constructed in Section~\ref{sec:OBScontr} for the dual system $(A^\ast,C^\ast,B^\ast,D^\ast)$ in all but two respects: $G_1^\ast$ has a block lower-triangular structure (instead of block upper-triangular structure), and the choice of $K_1^\ast$ is slightly different from the choice of $G_2$ in Section~\ref{sec:OBScontr}.
  However, as seen in the proof of Theorem~\ref{thm:OBScontr}, the properties of $(G_1,G_2)$ only affect the closed-loop stability by guaranteeing the exponential stabilizability of the block-operator pair ``$(A_s^N+\ga_2 I,B_s^N)$'' in Step~3 of the construction algorithm in Section~\ref{sec:OBScontr}. Because of duality, this property corresponds exactly
to the exponential detectability of the block operator pair ``$(C_s^N,A_s^N+\ga_2 I)$'' for the controller in the current theorem, and therefore the required stabilizability property is guaranteed by Lemma~\ref{lem:BlockOpStabDet}.
  Moreover, the definitions of the Galerkin approximation in Section~\ref{sec:Galbackground} imply that the approximation $( (A^\ast)^N,(C^\ast)^N,(B^\ast)^N)$ of the dual system $(A^\ast,C^\ast,B^\ast)$ is given by $(A^\ast)^N=(A^N)^\ast$, $(B^\ast)^N=(B^N)^\ast$, and $(C^\ast)^N=(C^N)^\ast$ with the same choices of the approximating subspaces $V^N$.
In addition, it is straightforward to check that the reduced order model constructed using balanced truncation for a dual system coincides with the dual system of the reduced order model of the original system, and the reduced dual system convergences in the graph topology to the dual of the original system.
Because of this, it follows from the proof of Theorem~\ref{thm:OBScontr} that $A_e^\ast + \ga I$ generates an exponentially stable semigroup when $N$ and $r$ are sufficiently large. Since $X_e$ is a Hilbert space, also $e^{\ga t}T_e(t)$  generated by $A_e + \ga I$ is exponentially stable.
\end{proof}

\section{Robust Controller Design for Parabolic PDE Models}
\label{sec:PDEcontrol}

In this section we apply the control design algorithms in Section~\ref{sec:Controllers} for selected PDE models. In each case we use two distinct Galerkin approximations, one (of order $N$) for constructing the controller and a second one (of order $n\gtrgtr N$) for simulating the behaviour of the original system.

\subsection{A 1D Reaction--Diffusion Equation}
\label{sec:1Dheat}

Consider a one-dimensional reaction--diffusion equation on the spatial domain $\Omega = (0,\,1)$ with distributed control and observation and Neumann boundary disturbance,
\begin{subequations}
\label{eq:HeatCon}
\begin{align}
\frac{\partial x}{\partial t} (\xi, t) &=  \frac{\partial}{\partial \xi}\left(\ga(\xi)  \frac{\partial x}{\partial \xi}(\xi,t) \right) + \gg(\xi) x(\xi,t)  + 
b(\xi)
u(t), \\
\frac{\partial x}{\partial \xi}(0, t) &= \wdist (t), \;    \frac{\partial x}{\partial \xi}(1, t) = 0, \; x(\xi, 0) = x_0 (\xi),  \\
y(t) &= \int_0^1 x(\xi,t)c(\xi)d\xi.
\end{align}
\end{subequations}
We assume $\ga\in W^{1,\infty}(0,1;\R)$ with $\ga(\xi) \ge \ga_0 >0$ for all  $\xi \in (0,1)$, $\gg\in \Lp[\infty](0,1;\R)$, and $b,c\in\Lp[2](0,1;\R)$.
The disturbance signal $\wdist(t)$ acts on the  left boundary.
The system~\eqref{eq:HeatCon} is a more general version of the 1D heat equation studied in~\cite{PhaPauMTNS18}.

Choose $X=\Lp[2](0,1)$.
Due to the boundary disturbance at $\xi=0$, the system~\eqref{eq:HeatCon} has the form of a \keyterm{boundary control system}~\citel{CurZwa95}{Sec. 3.3}, 
\eq{
  \dot{x}(t)&= \mc{A} x(t) + Bu(t) \\
  y(t)&= Cx(t)\\
  \wdist(t)&= \mc{B}_d x(t) 
}
where $\mc{A}x=\frac{\partial}{\partial\xi}(\ga(\cdot)\pd{x}{\xi})+\gg(\cdot)x$ for $x\in\Dom(\mc{A})=\setm{x\in H^2(0,1)}{x'(1)=0}$,
$B = b(\cdot)\in \Lin(\R,X)$, $C=\iprod{\cdot}{c(\cdot)}\in \Lin(X,\R)$,
and $\mc{B}_dx=x'(1)$ for $x\in \Dom(\mc{A})$.
The disturbance signal $\wdist(\cdot)$ is assumed to be of the form~\eqref{eq:yrefwdistwd} and is therefore smooth.
As in~\citel{CurZwa95}{Sec. 3.3, Ex. 3.3.5} we can make a change of variables $\tilde{x}(t)=x(t)-B_{d0} \wdist(t)$ where $B_{d0}\in \Lin(\R,X)$ is such that $\Dom(\mc{A})\subset \ran(B_{d0})$ and $\mc{B}_dB_{d0}=I$. This allows us to write the PDE system~\eqref{eq:plantintro} in the form
\eq{
  \dot{\tilde{x}}(t) &= A \tilde{x}(t) + Bu(t)+ \pmat{\mc{A}B_{d0},\, -B_{d0}}\pmat{\wdist(t)\\\wdistdot(t)}\\
  y(t)&= C \tilde{x}(t) 
+ \pmat{CB_{d0},\, 0} \pmat{\wdist(t)\\\wdistdot(t)}
}
where $Ax = \mc{A}x$ for $x\in \Dom(A) := \Dom(\mc{A})\cap \ker(\mc{B}_d)$. Since $B_d:=[\mc{A}B_{d0},\, -B_{d0}]\in \Lin(\R^2,X)$ and $(\wdist(t),\wdistdot(t))^T$ is of the form~\eqref{eq:yrefwdistwd}, this system is indeed of the form~\eqref{eq:plantintro} and the results in Section~\ref{sec:Controllers} are therefore applicable for~\eqref{eq:HeatCon}. Note that it is not necessary to compute the expressions of the operators $B_{d0}$, $\mc{A}B_{d0}$ and $CB_{d0}$ since the robustness of the controller implies that the disturbance signal is rejected for any disturbance input and feedthrough operators.

Now $\Dom(A)=\setm{x\in H^2(0,1)}{x'(0)=x'(1)=0}$ and if 
we choose $V = H^1(0,1)$ with inner product $\iprod{\phi}{\psi}_V=\int_0^1 \phi'(\xi)\overline{\psi'(\xi)}d\xi+ \int_0^1 \phi(\xi)\overline{\psi(\xi)}d\xi$,
then the operator $A$ is defined by the bounded and coercive sesquilinear 
form $a:V \times V\to \C$
\begin{align*}
  a(\phi, \psi) = \langle \ga(\cdot) \phi', \, \psi' \rangle + \langle \gg(\cdot) \phi, \, \psi \rangle  .
\end{align*}
  We assume $b(\cdot)$ and $c(\cdot)$ are such that $(A,B,C)$ is exponentially stabilizable and detectable, which in this case means that $\iprod{b}{\phi}\neq 0$ and $\iprod{\phi}{c}\neq 0$ for any eigenfunctions $\phi$ of $A$ associated to unstable eigenvalues~\citel{CurZwa95}{Sec. 5.2}.

  For the spatial discretization of~\eqref{eq:HeatCon} we use the Finite Element Method with piecewise linear basis functions. These approximations have the required property~\eqref{eq:Galerkinassumption} by~\cite{Cia78}.

\subsubsection*{A Simulation Example}
As a numerical example, we consider~\eqref{eq:HeatCon} with parameters
  \eq{
    \ga(\xi) = \frac{2-\xi}{4},
    \quad \gg(\xi) = 12 \xi,
    \quad b(\xi)=4\chi_{(.25,\, .5 )} (\xi),
    \quad c(\xi)=4\chi_{(.5,\, .75 )} (\xi) 
  }
where $\chi_{(a,b)}(\cdot)$ denotes the characteristic function on the interval $(a,b)$.
The control $u(t)$ and observation $y(t)$ act on the subintervals $(.25,.5)$ and $(.5,.75)$ of $\Omega$, respectively.
We consider the reference and disturbance signals 
  \eq{
    \yref(t) = \cos (t) + \frac{1}{2}\sin (2t) - 2 \cos (3t), \qquad
    \wdist (t) = \frac{1}{4} \sin (4t). 
  }
  The set of frequencies in~\eqref{eq:yrefwdist} in $\set{\gw_k}_{k=0}^q$ is  $\set{1,2,3,4}$ with $q=4$ and $n_k=1$ for all $k\in \List{4}$.
We modify the internal model in Section~\ref{sec:Controllers} in such a way that the parts associated to $\gw_0=0$ are omitted.

We construct the dual observer-based controller in Section~\ref{sec:DOBScontr}. 
In the absence of the frequency $0$ the internal model has dimension $\dim Z_0=p\times q\times 2=8$.
In the controller construction, we use a Finite Element approximation of order $N=300$. The parameters of the stabilization are chosen as 
\eq{
  \ga_1= 0, \;
  \ga_2= .95, \;
  Q_1= 
  Q_2= I_X, \;
  R_1= 
  R_2= 1\in\R
.  
}
Finally, we use balanced truncation with order $r=12$.
The system~\eqref{eq:HeatCon} is unstable with a finite number of eigenvalues with positive real parts. 

For the simulation of the original system~\eqref{eq:HeatCon} we use a Finite Element approximation of order $n = 1000$. Figure~\ref{fig:1Dheatspectrum} depicts parts of the spectrum of the original system, the closed-loop system without model reduction in the controller (i.e., with $r=N$), and the closed-loop system with model reduction of order $r=12$.
\begin{figure}[ht]  
  \begin{center}
\includegraphics[width=0.65\linewidth]{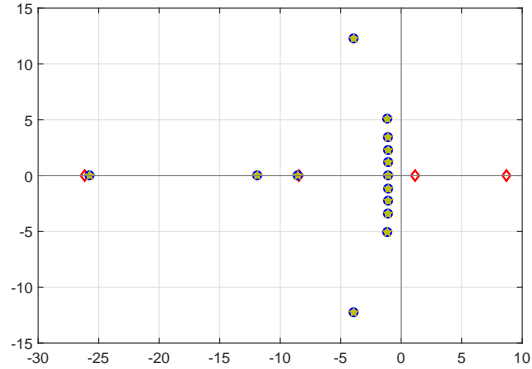}
\caption{Spectra of the uncontrolled system (red diamonds) and the closed-loop system with $r=N=300$ (yellow stars) and $r=12$ (blue circles).
    }
    \label{fig:1Dheatspectrum}
  \end{center}
\end{figure}

 The output of the controlled system for the initial states $x_0(\xi)=-\xi/10$ and $z_0=0\in \R^{8+12}$ of the system and the controller is depicted in Figure~\ref{fig:1Dheatoutput}.

\begin{figure}[ht]  
  \begin{center}
    \includegraphics[width=0.75\linewidth]{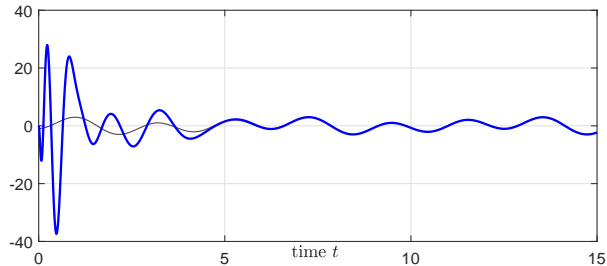}
    \caption{Output of the 1D heat equation with the dual observer-based controller.}
    \label{fig:1Dheatoutput}
  \end{center}
\end{figure}

\subsection{A 2D Reaction--Diffusion--Convection Equation}
\label{sec:2Dheat}
We consider a controlled reaction--diffusion--convection equation 
on a 2-dimensional bounded domain $\Omega \subset \R^2$ with $C^\infty$-smooth boundary $ \partial \Omega$ and assume
$\Omega$ is located locally on one side of $ \partial \Omega$. The PDE is defined as (see~\citel{BanKun84}{Sec. 3})
\begin{subequations}
\label{eq:ParaSys}
\begin{align}
  \frac{\partial x}{\partial t}(\xi,t) &=  \nabla( \ga(\xi) \nabla x(\xi,t) )+ \nabla \cdot (\gb(\xi) x(\xi,t) ) \\
  &\hspace{4cm} + \gg(\xi) x(\xi,t)    + f(\xi)+ B u(t), \\
  x (\xi,t) &= 0, \quad \mbox{on} ~ \xi\in\partial\Omega,  \quad x(\xi, 0) = x_0 (\xi) \\
  y(t)&= Cx(\cdot,t)
\end{align}
\end{subequations} 
with state $x:(0,\infty)\times \Omega\to \R$.
The possible source term $f(\xi)$ can be treated as a disturbance input with frequency $\gw_0=0$, and it will be handled by the internal model based controller.
  Here $\ga \in W^{1,\infty}(\Omega, \R)$ with $\ga(\xi)\geq \ga_0>0$ for all $\xi\in\Omega$,
  $\gb=(\gb_1(\cdot),\gb_2(\cdot))^T$ with $\gb_1,\gb_2\in W^{1,\infty}(\Omega;\R)$,
  and $\gg,~f \in L^\infty(\Omega, \R)$.
We assume~\eqref{eq:ParaSys} has $m$ distributed inputs and therefore $u(t)=(u_k(t))_{k=1}^m\in U=\R^m$ and
\eq{
  B u(t) =  \sum_{k=1}^m {u_k(t) b_k(\cdot)}
}
where $b_k(\cdot)\in \Lp[2](\Omega;\R)$ are fixed functions. Similarly we assume the system has $p$ measured outputs so that $y(t) = (y_k(t))_{k=1}^p\in Y=\R^p$ and 
\eq{
  y_k(t)=\int_\Omega x(\xi,t)c_k(\xi)d\xi
}
for some fixed $c_k(\cdot)\in \Lp[2](\Omega;\R)$. 

The system~\eqref{eq:ParaSys} can be written in the form~\eqref{eq:plantintro} on $X=\Lp[2](\Omega;\R)$. If we choose $V=H_0^1(\Omega,\R)$, then the system operator $A$ is determined by the sesquilinear form $a:V\times V\to \C$ such that for all $\phi,\psi\in V$,
\eq{
  a(\phi, \psi) &=
  \iprod{\ga\nabla \phi}{\nabla\psi}_{\Lp[2]}+
  \iprod{\gb \phi}{\nabla\psi}_{\Lp[2](\Omega;\R^2)}
  +\iprod{\gg\phi}{ \psi}_{\Lp[2]}   .
}
Similarly as in~\citel{BanKun84}{Sec. 3} we can deduce that $a(\cdot,\cdot)$ is bounded and coercive.
The input and output operators $B\in \Lin(U,X)$ and $C\in \Lin(X,Y)$ are such that $Bu=\sum_{k=1}^m b_k(\cdot)u_k$ for all $u=(u_k)_k\in U$ and $Cx = \left( \int_\Omega x(\xi)c_k(\xi)d\xi \right)_{k=1}^p$ for all $x\in X$.
  We assume $\set{b_k(\cdot)}_{k=1}^m$ and $\set{c_k(\cdot)}_{k=1}^p$ are such that $(A,B,C)$ is exponentially stabilizable and detectable.
The autonomous source term  $f(\xi)$ is considered as a disturbance input, i.e., we write $f(\cdot)=B_d\wdist(t)$ where $\wdist(t)\equiv 1$ and $B_d=f(\cdot)\in \Lin(\R,X)$.

To discretize the equation using Finite Element method, the domain $\Omega$ is approximated with a polygonal domain $\Omega_D$ and we consider a partition of $\Omega_D$ into non-overlapping triangles. The approximating subspaces $V^N$ are chosen as the span of $N$ piecewise linear hat functions $\phi_k$. The subspaces $V^N$ then have the required property~\eqref{eq:Galerkinassumption} by~\cite{Cia78}.  

\begin{rem}
  Also in the case of the 2D reaction--diffusion--convection equation it would be in addition possible to consider boundary disturbances using the same approach as in Section~\ref{sec:1Dheat}.
\end{rem}

\subsubsection*{A Simulation Example}
As a particular numerical example, we consider
a reaction--diffusion--convection equation on the unit disk $\Omega = \{\xi = (\xi_1, \xi_2) \in \R^2 \mid \xi_1^2 + \xi_2^2 <1\}$ with parameters
\eq{
  \ga(\xi) &= \frac{1}{2},\quad
  \gb(\xi) = \pmat{\cos (\xi_1) - \sin (2 \xi_2)
    \\ \sin (3\xi_1) + \cos (4 \xi_2)},\quad
  \gg(\xi) = 10,
  \quad f  = 0. 
}
We consider~\eqref{eq:ParaSys}  with two inputs and two measurements acting on rectangular subdomains of $\Omega$. More precisely, 
\eq{
  b_1(\cdot)=\chi_{\Omega_1},\;
  b_2(\cdot)=\chi_{\Omega_2},\;
  c_1(\cdot)=\chi_{\Omega_3},\;
  c_2(\cdot)=\chi_{\Omega_4}
}
where
$\Omega_1 = \left(\frac{3}{20},\frac{7}{20} \right) \times \left(\frac{1}{15},\frac{4}{15} \right)$ and $\Omega_2 = \left(\frac{3}{5},\frac{4}{5} \right) \times \left(-\frac{2}{25},\frac{2}{25} \right)  $, 
$\Omega_3 = \left(-\frac{7}{10},-\frac{1}{2} \right) \times \left(-\frac{29}{60},-\frac{11}{60} \right)$, and $\Omega_4 = \left(-\frac{1}{2},-\frac{3}{10} \right) \times \left(\frac{7}{25},\frac{13}{25} \right)$ .   
The configuration of the control inputs and measurements is illustrated in Figure~\ref{fig:2DheatIO}.

\begin{figure}[ht]  
  \hspace{-1cm}
  \begin{minipage}{0.51\linewidth}
    \begin{center}
      \includegraphics[width=.65\linewidth]{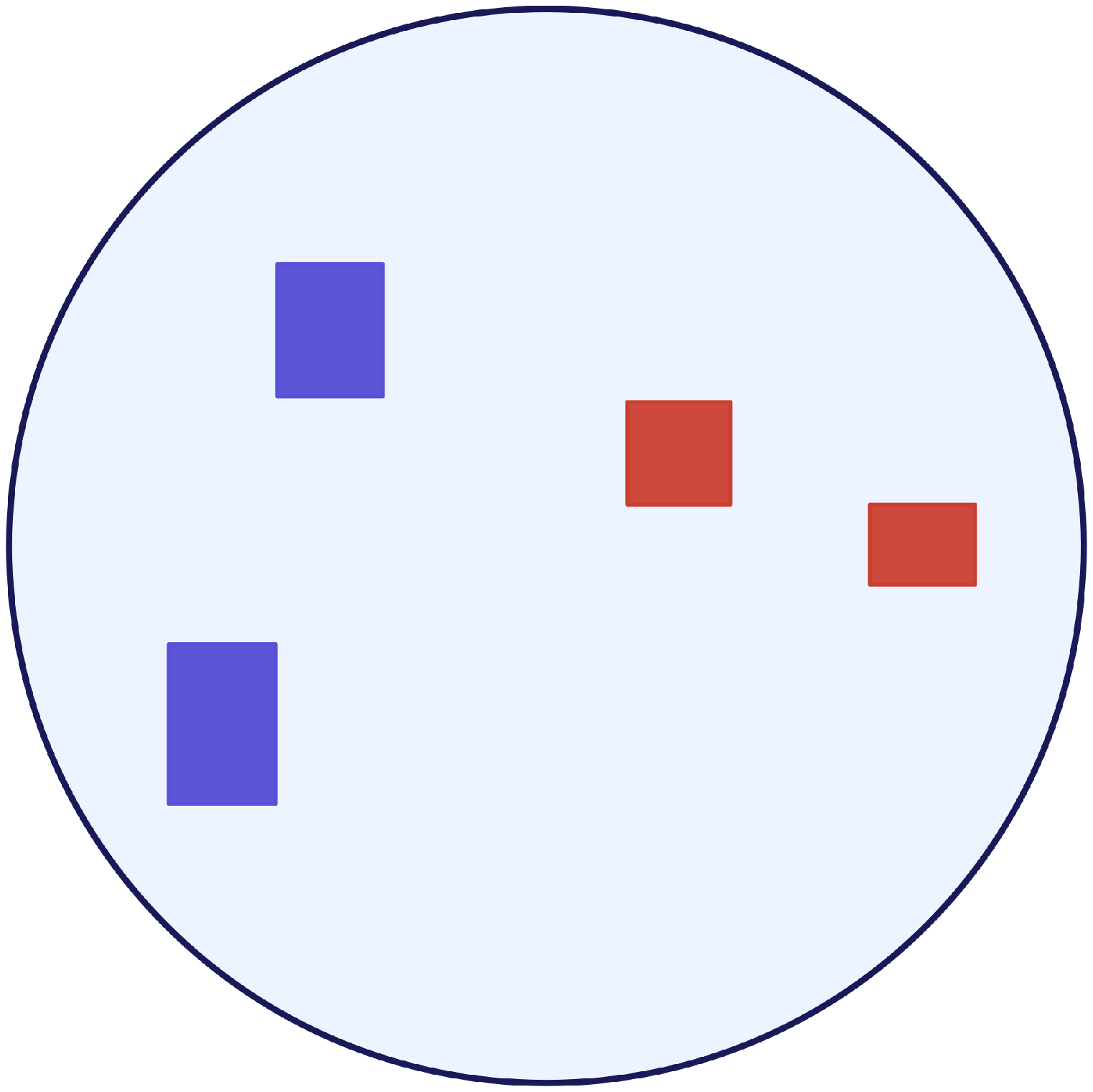} 
      \caption{\small Regions of control (red) and observation (blue).}
    \label{fig:2DheatIO}
    \end{center}
  \end{minipage}
  \hspace{-.5cm}
  \begin{minipage}{0.58\linewidth}
      \includegraphics[width=1\linewidth]{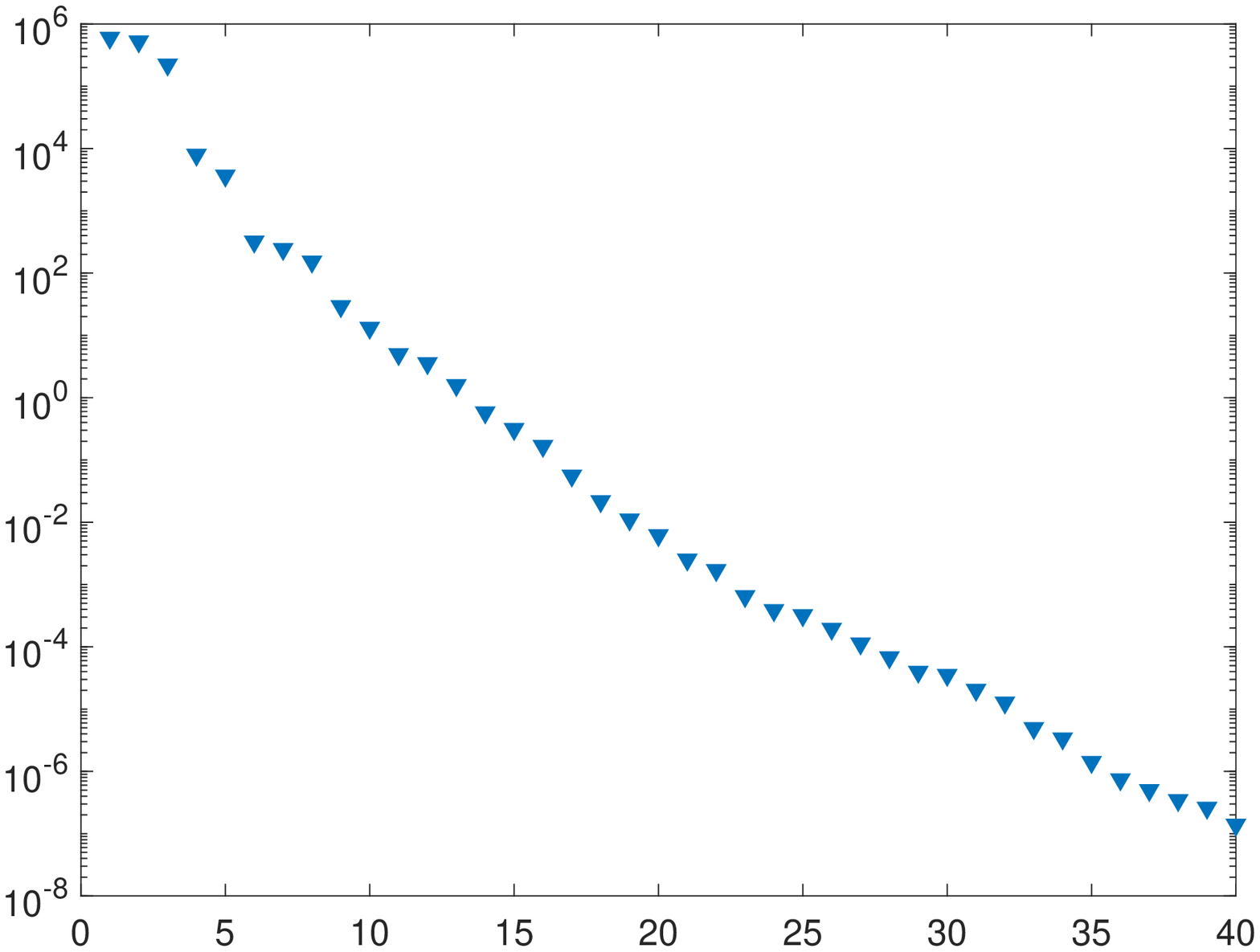}
      \caption{\small Hankel singular values of the Galerkin approximation.} \label{fig:2DheatHankel}
  \end{minipage}
\end{figure}

Our aim is to track a reference signal
\begin{align*}
  \yref(t) &= \pmat{20\cos (t) + 5\sin (2t) - 2 \cos (3t) \\ 45 \sin (10t) - 2\cos (t)}.
\end{align*}
The corresponding set of frequencies in~\eqref{eq:yrefwdist} is $\set{1,2,3,10}$ with $q=4$ and $n_k=1$ for all $k\in \List{4}$. We modify the internal model in Section~\ref{sec:Controllers} in such a way that the parts associated to $\gw_0=0$ are omitted.
We construct the dual observer-based controller in Section~\ref{sec:DOBScontr} using a Galerkin approximation with order $N=1258$ and subsequent balanced truncation with order $r=40$. 
In the absence of the frequency $0$, the internal model has dimension $\dim Z_0=p\times q\times 2=16$.
The FEM discretization is implemented using the Matlab PDE Toolbox functions. 
The parameters of the stabilization are chosen as 
\eq{
  \ga_1= 2, \;
  \ga_2= 2.5, \;
  Q_1= 
  Q_2= I_X, \;
  R_1= 
  R_2= 1\in\R
.  
}
The first Hankel singular values of the Galerkin approximation are plotted for illustration in Figure~\ref{fig:2DheatHankel}.

  In the simulation the original PDE is represented by
   another Finite Element approximation of~\eqref{eq:ParaSys} with order $n = 2072$. 
Figure~\ref{fig:2DheatSpectrum} depicts parts of the spectrum of the uncontrolled system and the closed-loop system. 
  In the plotted region the locations of the closed-loop eigenvalues for the controller without model reduction (i.e., with $r=N$) are very close to those with the final controller.

\begin{figure}[ht]  
  \begin{center}
    \includegraphics[width=0.65\linewidth]{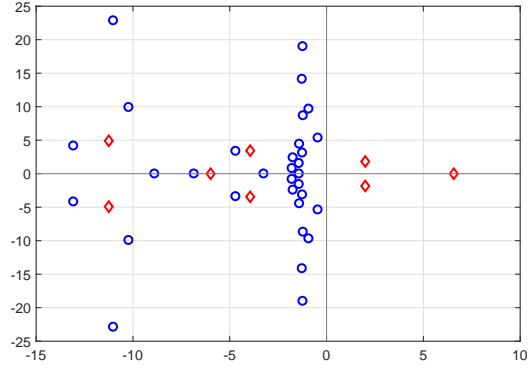}
    \caption{Spectra of the uncontrolled system (red diamonds) and the closed-loop system with $N=1258$ and $r=12$ (blue circles).}
    \label{fig:2DheatSpectrum}
  \end{center}
\end{figure}

The output of the controlled system for the initial states $x_0(\xi)=\cos(5\xi_1)$ and $z_0= 0\in\R^{16+40}$ of the system and the controller is depicted in Figure~\ref{fig:2DheatOutput}.
\begin{figure}[ht]  
  \begin{center}
    \includegraphics[width=.75\linewidth]{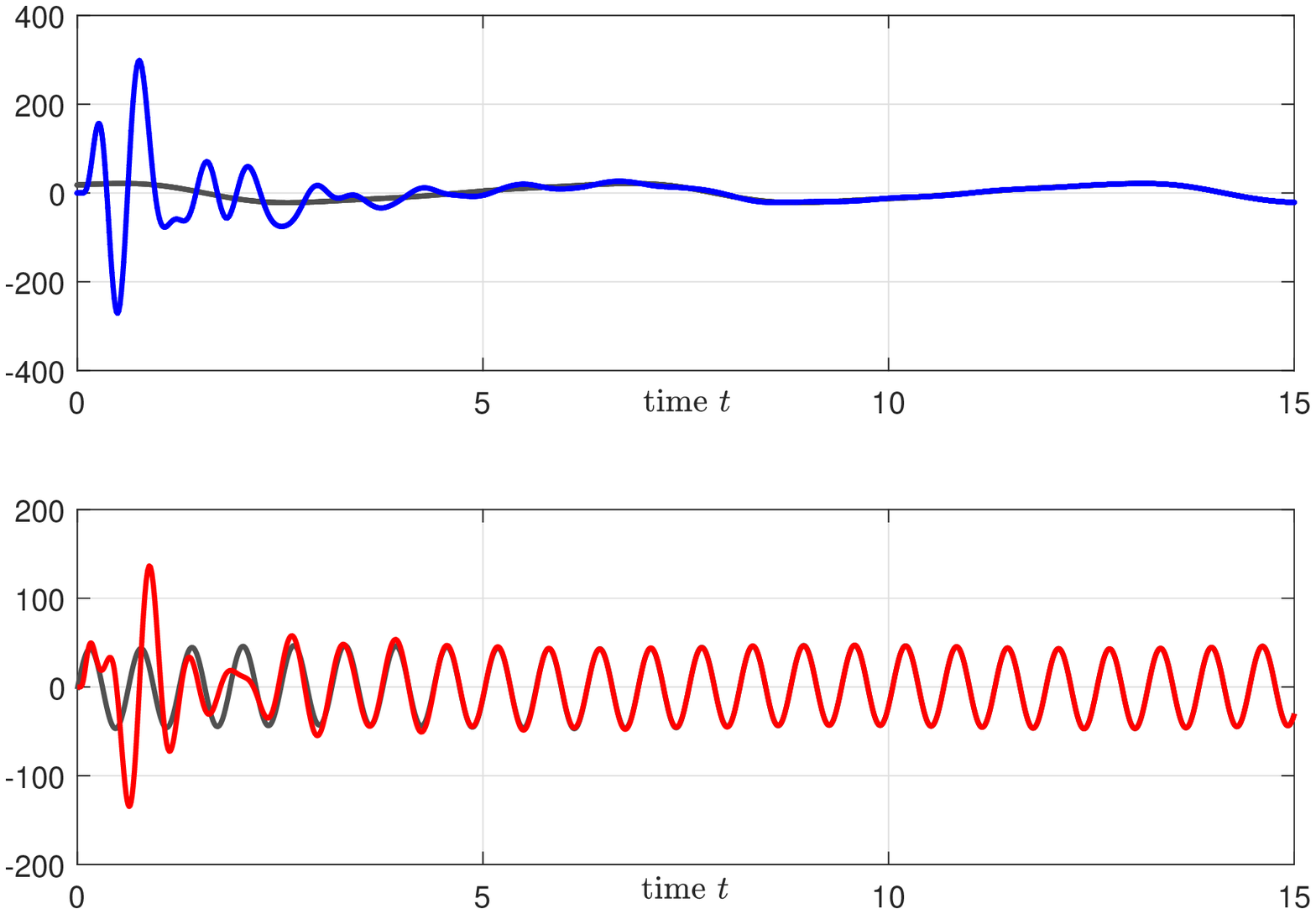}
    \caption{Output $y(t)=(y_1(t),y_2(t))^T$ of the system~\eqref{eq:ParaSys} with the dual observer-based controller (top: $y_1(t)$, bottom: $y_2(t)$).}
    \label{fig:2DheatOutput}
  \end{center}
\end{figure}

\subsection{A Beam Equation with Kelvin--Voigt Damping}
\label{sec:1Dbeam}

Consider a one-dimensional Euler-Bernoulli beam model on $\Omega = (0, \ell)$~\citel{ItoMor98}{Sec. 3}
\begin{subequations}
    \label{eq:EBmodel}
  \eqn{
    \frac{\partial^2 v}{\partial t^2}(\xi,t)  &+  \frac{\partial^2 }{\partial \xi^2}\left( \ga \frac{\partial^2 v}{\partial \xi^2}(\xi,t) + \gb \frac{\partial^3 v}{\partial \xi^2 \partial t}(\xi,t)  \right)
    \\
    &\hspace{1cm} + \gg \frac{\partial v}{\partial t}(\xi,t)  =  B_0u(t) + B_{d0}(\xi) \wdist(t) ,\\
    v(\xi,0)&=v_0(\xi), \qquad \pd{v}{t}(\xi,0)=v_1(\xi),\\
    y(t)&=C_1 v(\cdot,t)+C_2\dot{v}(\cdot,t)
  }
\end{subequations}
where $\ga,\gb,\gg\in\R$ are constants so that $\ga,\gb>0$ and $\gg\geq 0$. 
The input operator is defined by $B_0u=\sum_{k=1}^m b_k(\cdot)u_k$ for $u=(u_k)_{k=1}^m\in U=\R^m$ for some fixed $b_k(\cdot)\in \Lp[2](0,\ell)$
and the disturbance input operator $B_{d0}$ is defined analogously. 
The assumptions on the measurement operators for the deflection $v(\cdot,t)$ and velocity $\dot{v}(\cdot,t)$ are given later.

We consider a situation where the beam is clamped at $\xi = 0$ and free at $\xi=\ell$. The boundary conditions are
\begin{align*}
v(0,t) &= 0,  &&\frac{\partial v}{\partial \xi} (0,t) = 0, \\ 
\left[  \ga \frac{\partial^2 v}{\partial \xi^2} + \gb \frac{\partial^3 v}{\partial \xi^2 \partial t} \right]_{\xi = \ell} &= 0, \quad 
&&\left[  \ga \frac{\partial^3 v}{\partial \xi^3} + \gb \frac{\partial^4 v}{\partial \xi^3 \partial t} \right]_{\xi = \ell} = 0.
\end{align*}

Let $V_0 = \setm{   v \in H^2 (0,\ell) }{ v(0) = v'(0) = 0  }$.
We define an inner product on $V_0$ by
\begin{align*}
\langle \phi_1, \phi_2 \rangle_{V_0} =  \int_0^\ell  \phi_1''(\xi)  \phi_2''(\xi)d\xi, \qquad \forall\phi_1,\phi_2\in V_0.
\end{align*}
Defining the state as $x(t) = ( v(\cdot,t), \dot{v}(\cdot,t) )^T$ the beam model~\eqref{eq:EBmodel} can be written in the form~\eqref{eq:plantintro} on  $X=V_0\times \Lp[2](0,\ell)$ with  
\eq{
  &A = \pmat{ 0 & I \\ -\ga \frac{d^4}{d\xi^4} &  -\gb \frac{d^4}{d\xi^4} - \gg }, \; 
  B = \pmat{0 \\ B_0}, \; 
  B_d = \pmat{0 \\ B_{d0}}\\
  &\Dom(A)= \setm{(v,w)\in V_0\times V_0}{\ga v''+\gb w''\in H_r^2(0,\ell) }
}
where $H_r^2(0,\ell) = \setm{f\in H^2(0,\ell)}{f(\ell)=f'(\ell)=0}$.
We assume the measurement operators $C_1  \in \Lin(V_0,Y)$ and $C_2\in \Lin(\Lp[2](0,\ell),Y)$ for $Y=\R^p$, and thus $C_1v=(C_1^kv)_{k=1}^p$ where $C_1^k\in \Lin(V_0,\R)$  and  $C_2w= (\iprod{w}{c_k^1}_{\Lp[2]})_{k=1}^p$
for some fixed functions $c_k^1(\cdot)\in \Lp[2](0,\ell)$.
Since for any $0<\xi_0\leq \ell$ the point evaluation $C_{\xi_0}v=v(\xi_0)$ is a linear functional on $V_0$, it is in particular possible to consider pointwise tracking of the deflection with $y(t)=v(\xi_0,t)$ in~\eqref{eq:EBmodel}.

Choose $V=V_0\times V_0$. As shown in~\citel{ItoMor98}{Sec. 3} the operator $A$ is defined by a bounded and coercive sesquilinear form $a:V\times V\to \C$ defined so that for all $\phi=(\phi_1,\phi_2)\in V $ and $\psi=(\psi_1,\psi_2)\in V$ we have
\begin{align*}
a (\phi, \psi) 
= -\langle \phi_2, \psi_1  \rangle_{V_0} +   \langle \ga \phi_1 + \gb \phi_2, \psi_2 \rangle_{V_0}  + \gg  \iprod{\phi_2 }{ \psi_2 }_{\Lp[2]} .
\end{align*}

As the Galerkin approximation of~\eqref{eq:EBmodel} we use the Finite Element Method with cubic Hermite shape functions to approximate functions of $V_0$ and $\Lp[2](0,\ell)$ in the spaces $V_0^N$. As shown in~\citel{ItoMor98}{Sec. 3} the approximating subspaces $V^N = V_0^N\times V_0^N$ have the required property~\eqref{eq:Galerkinassumption}.
For additional details on the approximations, see~\citel{XiaBas99}{Sec. 4}.

\subsubsection*{A Simulation Example}

For a numerical example we consider a beam model with $\ell = 7$, $\ga = 0.5$, $\gb= 1$, and $\gg = 2$.
Similarly as in~\citel{ItoMor98}{Sec. 3} we choose $U=\R$ and $B_0=b_1(\cdot)$ with $b_1(\xi)=\xi$, and choose a measurement
\eq{
  y(t) = \int_5^6{v(\xi,t) d\xi},\qquad  \mbox{i.e.,} \quad C_1=\chi_{(5,6)}(\cdot), ~ C_2=0.
}
The disturbance $\wdist(t)$ acts on the interval $(3,6)$ so that $B_{d0}=\chi_{(3,6)}(\cdot)\in \Lin(\R,\Lp[2](0,\ell))$.

With our choices of parameters the damping in the beam model~\eqref{eq:EBmodel} is strong enough to stabilize the system exponentially. 
However, the stability margin of the system is very small.
In such a situation the finite-dimensional low-gain robust controllers~\cite{HamPoh00,RebWei03} 
typically only achieve very limited closed-loop stability margins and slow convergence of the output. In this example we use our controller design to improve the degree of stability of the original model and achieve an improved closed-loop stability margin.

We take the reference signal and disturbance signal 
\begin{align*}
  \yref(t) &= 3 \cos (t)- 2 \cos (3t)  + 15 \sin (5 t)  - 6 \sin (10t), \\
\wdist(t) &= 3 \sin(4t) + 5 \sin (7t).  
\end{align*}
The corresponding set of frequencies in~\eqref{eq:yrefwdist} is $\set{1,3,4,5,7,10}$ with $q=6$ and $n_k=1$ for all $k\in \List{6}$. We modify the internal model in Section~\ref{sec:Controllers} in such a way that the parts associated to $\gw_0=0$ are omitted.
We construct the observer-based controller in Section~\ref{sec:DOBScontr} using a Galerkin approximation with order $N=58$ and subsequent balanced truncation with order $r=10$. 
In the absence of the frequency $0$, the internal model has dimension $\dim Z_0=p\times q\times 2=12$.

  The stability margins in the stabilizability of $(A,B)$ and the detectability of $(C,A)$
are limited because the beam model~\eqref{eq:EBmodel} is known to have an accumulation point of eigenvalues at $\glacc\in\R_-$~\cite{ZhaGuo11}. In particular, the assumptions of the detectability of $(C,A+\ga_1 I)$ and the stabilizability of $(A+\ga_2 I,B)$ can only be satisfied if $0\leq \ga_1,\ga_2< \abs{\glacc}$.
To find the upper bound $\abs{\glacc}$, the spectrum of $A$ can be computed similarly as in~\citel{LuoGuo99book}{Sec. 4.3}. 
In particular,
the eigenvalues $\gl_n\neq -\ga/\gb$ of $A$ are solutions of the quadratic equation $\gl_n^2 -(\gb\eta_n+\gg)\gl_n-\ga\eta_n = 0$, where $\eta_n\in\R_+$ are such that $\phi_n'''' = \eta_n\phi_n$ for some $\phi_n(\cdot)\in H^4(0,\ell)$ satisfying $\phi_n(0)=\phi_n'(0)=\phi_n''(\ell)=\phi_n'''(\ell)=0$. Since $\eta_n\to \infty$ as $n\to \infty$,
a direct computation shows that the eigenvalues $\gl_n$ have a limit 
$\gl_n\to -\ga/\gb=:\glacc$ as $n\to \infty$. Thus the condition on $\ga_1$ and $\ga_2$ becomes $0\leq \ga_1,\ga_2<\abs{\glacc}=\ga/\gb=0.5$.
Motivated by this, the parameters of the stabilization are chosen as 
\eq{
  \ga_1= 
  \ga_2= 0.4, \;
  Q_1= 
  Q_2= I_X, \;
  R_1= 10^{-3},\;
  R_2= 10^3
.  
}

For the simulation of the original system~\eqref{eq:ParaSys} we use another Finite Element approximation of order $n = 140$. 
Figure~\ref{fig:1DbeamSpectrum} depicts parts of the spectrum of the uncontrolled system and the closed-loop system.
  In the plotted region the locations of the closed-loop eigenvalues for the controller without model reduction (i.e., with $r=N$) are very close to those with the final controller.

\begin{figure}[ht]  
  \begin{center}
    \includegraphics[width=.65\linewidth]{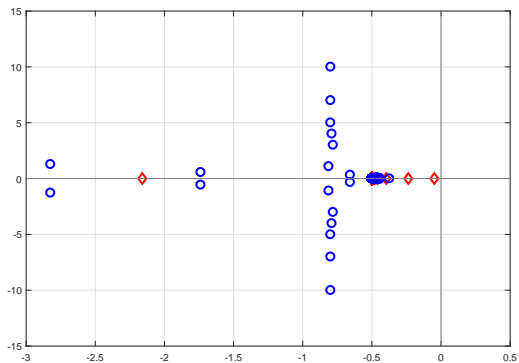}
    \caption{Spectra of the uncontrolled system (red diamonds) and the closed-loop system with $N=58$ and $r=10$ (blue circles).} 
    \label{fig:1DbeamSpectrum}
  \end{center}
\vspace{-.3cm}
\end{figure}

The output of the controlled system for the initial states $x_0(\xi)=\cos(5\xi_1)-2$ and $z_0= -3\cdot \bm{1}\in\R^{12+10}$ of the system and the controller is depicted in Figure~\ref{fig:2DheatOutput}.

\vspace{-.3cm}

\begin{figure}[ht]  
  \begin{center}
    \includegraphics[width=.75\linewidth]{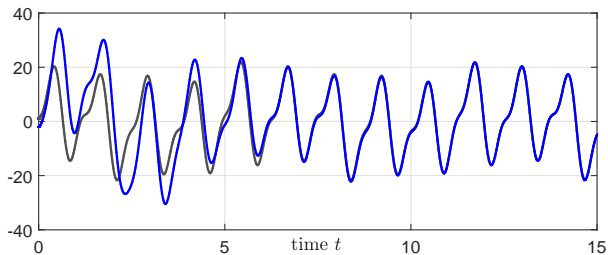}
    \caption{Output (blue) of the beam model with the observer-based controller and the reference signal (gray).}
    \label{fig_tracking1d_beam}
  \end{center}
\vspace{-.4cm}
\end{figure}

\section{Conclusions}
\label{sec:Conclusions}

We have studied the construction of finite-dimensional low-order controllers for robust output regulation of parabolic PDEs and other infinite-dimensional systems with analytic semigroups. 
We have presented two controller structures constructed using a Galerkin approximation of the control system and balanced truncation.
Theorems~\ref{thm:OBScontr} and~\ref{thm:DOBScontr} guarantee that the controllers achieve robust output tracking and disturbance rejection provided that the orders $N$ and $r\leq N$ of the Galerkin approximation and the model reduction, respectively, are sufficiently high, but the methods used in the proofs do not provide any concrete bounds for the sizes of $N$ and $r$. The rate of decay of the Hankel singular values can be used together with Lemma~\ref{lem:ROMconv} as a rough indicator of how much reduction is possible in the last step of the controller construction algorithm. Deriving precise and reliable lower bounds $N$ and $r$ to guarantee closed-loop stability is an important topic for future research. 
Another open question is to develop a way to reliably estimate the stability margin of the closed-loop system for particular orders $N$ and $r$.

\appendix
\enlargethispage{1ex}

\section{Additional Lemmata}
\label{sec:Appendix}

\begin{lem}
  \label{lem:CLvsOLconv}
  The system $(A^n,B^n,C^n)$ converges to $(A,B,C)$ in the graph topology if and only if for some $Q\in \Lin(Y,U)$ the system $(A^n+B^nQC^n,B^n,C^n)$ converges to $(A+BQC,B,C)$ in the graph topology.
\end{lem}

\begin{proof}
The result follows from the property that in the graph topology the convergence of open loop systems is equivalent to the convergence of closed-loop systems~\cite[Prop. 7.3.40]{Vid85}, \cite[Thm. 3.3]{Zhu89}.
\end{proof}

\begin{lem}
  \label{lem:BlockOpStabDet}
  Suppose Assumption~\textup{\ref{ass:Psurj}} and the standing assumptions are satisfied and $G_1$ is as in Sections~\textup{\ref{sec:OBScontr}} and~\textup{\ref{sec:DOBScontr}}. 
  Let $\ga \geq 0$ be such that $(A+\ga I,B,C)$ is exponentially stabilizable and detectable.
  Then the following hold.
\begin{itemize}
  \item[\textup{(a)}] 
    In the case of the observer-based controller, the pair
  \eqn{
    \label{eq:BlockOpPair1}
    \left(\pmat{G_1&G_2C\\0&A}+\ga I,\pmat{G_2D\\B}  \right)
  }
  is exponentially stabilizable.
\item[\textup{(b)}] In the case of the dual observer-based controller, the pair
  \eqn{
    \label{eq:BlockOpPair2}
    \left(\pmat{DK_1,\; C},\pmat{G_1&0\\BK_1&A} + \ga I  \right)
  }
  is exponentially detectable.
\item[\textup{(c)}] 
  If $K=[K_1,K_2]$ stabilizes the pair~\eqref{eq:BlockOpPair1}, then $(K_1,G_1)$ is observable.
  If $\mc{G}_2=\pmatsmall{G_2\\L}$ stabilizes the pair~\eqref{eq:BlockOpPair2}, then $(G_1,G_2)$ is controllable.
\end{itemize}
\end{lem}

\begin{proof}
  We can assume $\ga=0$, since otherwise we may consider $\tilde{A}:= A+\ga I$ and $\tilde{G}_1:= G_1+\ga I$.
  We begin by proving part (b). Due to our assumptions we can choose $L_1$ so that $A+L_1C$ is exponentially stable and $P_L(i\gw_k)=CR(i\gw_k,A_L)B_L+D$ with $A_L=A+L_1C$ and $B_L=B+L_1D$ is surjective for every $k\in\List[0]{q}$.
  Choose $L=L_1+HG_2$ where $H$ is the unique solution of the Sylvester equation $HG_1=A_L H+B_L K_1$ and $G_2\in \Lin(Y,Z_0)$ is such that the matrix $G_1+G_2(CH+DK_1)$ is Hurwitz. 
  The choice of $G_2$ is possible provided that the pair $(CH+DK_1,G_1)$ is observable. To see that this is true, let $k\in \List[0]{q}$ and $0\neq \phi_k\in \ker(\pm i\gw_k-G_1)$. Since $H$ is the solution of the Sylvester equation and $G_1$ and $K_1$ have special structure, we have $\phi_k = (\phi_k^0,\pm i\phi_k^0,0,\ldots,0)^T$, $H\phi_k = R(\pm i\gw_k,A_L)B_LK_1^{k1}\phi_k^0$ and 
\eq{
  (CH+DK_1)\phi_k
  = (CR(\pm i\gw_k,A_L)B_L+D)K_1^{k1}\phi_k^0\neq 0 
}
by the choices of $K_1^{k1}\in \Lin(Y,U)$.
Thus the pair $(CH+DK_1,G_1)$ is observable.
A direct computation then shows that 
  \eq{
    \MoveEqLeft \pmat{I&0\\H&-I}  \left(\pmat{G_1&0\\BK_1&A}+ \pmat{G_2\\L}\pmat{DK_1, \; C}\right)\pmat{I&0\\H&-I}\\
    &= \pmat{G_1+G_2(CH+DK_1)&-G_2C\\0&A+L_1C},
  }
  which generates an exponentially stable semigroup.

Part (a) can be proved analogously by considering adjoint operators.
To prove (c), assume $K=[K_1,K_2]$ stabilizes the pair~\eqref{eq:BlockOpPair1}. If $(K_1,G_1)$ is not controllable, there exist $k\in \List[0]{q}$ and $0\neq\phi_k \in \ker(\pm i\gw_k-G_1)$ such that $K_1\phi_k=0$. Then we also have
  \eq{
    \left(\pmat{G_1&G_2C\\0&A}+\pmat{G_2D\\B}\pmat{K_1,K_2}  \right) \pmat{\phi_k\\0} = \pm i\gw_k \pmat{\phi_k\\0},
  }
  which contradicts the assumption that $[K_1,K_2]$ stabilizes~\eqref{eq:BlockOpPair1}.
The second claim follows similarly by considering adjoint operators.
\end{proof}

\subsection*{Acknowledgement}

The authors would like to express their gratitude to Petteri Laakkonen for providing background information on model reduction methods and convergence of transfer functions, and for carefully examining the manuscript.

\end{document}